\documentclass[12pt]{article}

\usepackage{amsfonts,amssymb,amsmath,amsthm,epsfig,euscript}

\usepackage{tikz}
\usetikzlibrary{matrix,trees,arrows}

\usetikzlibrary{positioning}
\usetikzlibrary{fit}
\usetikzlibrary{patterns}

\newtheorem{theorem}{Theorem}
\newtheorem{lemma}[theorem]{Lemma}

\theoremstyle{definition}
{

\newtheorem{problem}{Problem}
}

\long\def\symbolfootnote[#1]#2{\begingroup
\def\thefootnote{\fnsymbol{footnote}}\footnote[#1]{#2}\endgroup}

\newcommand{\red}[1][\sigma]{\mathrm{red}(#1)}

\newcommand{\sg}{\sigma}

\newcommand{\mmp}{\mathrm{mmp}}
\newcommand{\MMP}{\mathrm{MMP}}

\def\A{\mathcal{A}}

\newcommand{\fig}[2]{\begin{figure}[ht]
\centerline{\scalebox{.66}{\epsfig{file=#1.eps}}}
\caption{#2}
\label{fig:#1}
\end{figure}}

\newcommand{\shadetheboxes}[1]{
	\foreach \x/\y in {#1}
      	\fill[pattern color = black!65, pattern=north east lines] (\x,\y) rectangle +(1,1);
	}
	
\newcommand{\drawthegrid}[1]{
	\draw (0.01,0.01) grid (#1+0.99,#1+0.99);
	}
	
\newcommand{\drawtheclpattern}[1]{
	\foreach \x/\y in {#1}
      	\filldraw (\x,\y) circle (6pt);
	}

\newcommand{\drawspecialbox}[1]{
	\foreach \x/\y/\z/\w/\A in {#1}
		{
       		\fill[color = white!100, opacity=1, rounded corners = 1.5pt] (\x+0.125,\y+0.125) rectangle (\z-0.125,\w-0.125);
       		\draw[color = black, rounded corners = 1.5pt] (\x+0.125,\y+0.125) rectangle (\z-0.125,\w-0.125);
       		\fill[black] (\x/2+\z/2,\y/2+\w/2) node {$\scriptstyle\A$};
       	}
    }

\newcommand{\mmpattern}[5]{									
  \raisebox{0.6ex}{
  \begin{tikzpicture}[scale=0.35, baseline=(current bounding box.center), #1]
  \useasboundingbox (0.0,-0.1) rectangle (#2+1.4,#2+1.1);
    
    \shadetheboxes{#4}
    
    \drawthegrid{#2}
    
    \drawspecialbox{#5}
    
    \drawtheclpattern{#3}

  \end{tikzpicture}}
}

\title{Quadrant marked mesh patterns in $132$-avoiding permutations II}

\author{
Sergey Kitaev \\
\small University of Strathclyde\\[-0.8ex]
\small Livingstone Tower, 26 Richmond Street\\[-0.8ex]
\small Glasgow G1 1XH, United Kingdom\\[-0.8ex]
\small \texttt{sergey.kitaev@cis.strath.ac.uk}
\and
Jeffrey Remmel \\
\small Department of Mathematics\\[-0.8ex]
\small University of California, San Diego\\[-0.8ex]
\small La Jolla, CA 92093-0112. USA\\[-0.8ex]
\small \texttt{jremmel@ucsd.edu}
\and
Mark Tiefenbruck\\[-0.8ex]
\small Department of Mathematics\\[-0.8ex]
\small University of California, San Diego\\[-0.8ex]
\small La Jolla, CA 92093-0112. USA\\[-0.8ex]
\small \texttt{mtiefenb@math.ucsd.edu}
}

\date{\small Submitted: Date 1;  Accepted: Date 2;
 Published: Date 3.\\
\small MR Subject Classifications: 05A15, 05E05}

\begin{document}
\maketitle

\begin{abstract}
\noindent 
Given a permutation $\sg = \sg_1 \ldots \sg_n$ in the symmetric group 
$S_n$, we say that $\sg_i$ matches the marked mesh pattern 
$MMP(a,b,c,d)$ in $\sg$ if there are at least 
$a$ points to the right of $\sg_i$ in $\sg$ which are greater than 
$\sg_i$, at least $b$ points to the left of $\sg_i$ in $\sg$ which 
are greater than $\sg_i$,  at least $c$ points to the left of 
$\sg_i$ in $\sg$ which are smaller  than $\sg_i$, and 
at least  $d$ points to the right of $\sg_i$ in $\sg$ which 
are smaller than $\sg_i$.

This paper is continuation of the systematic study of the distribution 
of quadrant marked mesh patterns in 132-avoiding permutations 
started in \cite{kitremtie} where 
we mainly studied the distribution of the number of matches 
of $MMP(a,b,c,d)$ in 132-avoiding permutations 
where exactly one of $a,b,c,d$ is greater 
than zero and the remaining elements are zero. In this paper, 
we study  the distribution of the number of matches 
of $MMP(a,b,c,d)$ in 132-avoiding permutations 
where exactly two of $a,b,c,d$ are greater 
than zero and the remaining elements are zero. 
We provide explicit recurrence relations to enumerate our objects which 
can be used to give closed forms for the generating functions associated 
with such distributions. In many  cases, we provide combinatorial explanations of the coefficients that appear in our generating functions. The case of quadrant marked mesh patterns $MMP(a,b,c,d)$ where three or more of  $a,b,c,d$ are 
constrained to be greater than 0 will be studied in \cite{kitremtieIII}.\\

\noindent {\bf Keywords:} permutation statistics, quadrant marked mesh pattern, distribution, Pell numbers 
\end{abstract}

\tableofcontents

\section{Introduction}

The notion of mesh patterns was introduced by Br\"and\'en and Claesson \cite{BrCl} to provide explicit expansions for certain permutation statistics as, possibly infinite, linear combinations of (classical) permutation patterns.  This notion was further studied in \cite{AKV,HilJonSigVid,kitlie,kitrem,kitremtie,Ulf}.

Kitaev and Remmel \cite{kitrem} initiated the systematic study of distribution of quadrant marked mesh patterns on permutations. The study was extended to 132-avoiding permutations by Kitaev, Remmel and Tiefenbruck in \cite{kitremtie}, and the present paper continues this line of research. 
Kitaev and Remmel also studied the distribution of quadrant marked 
mesh patterns in up-down and down-up permutations \cite{kitrem2,kitrem3}. 

Let $\sigma = \sg_1 \ldots \sg_n$ be a permutation written in one-line notation. Then we will consider the 
graph of $\sg$, $G(\sg)$, to be the set of points $(i,\sg_i)$ for 
$i =1, \ldots, n$.  For example, the graph of the permutation 
$\sg = 471569283$ is pictured in Figure 
\ref{fig:basic}.  Then if we draw a coordinate system centered at a 
point $(i,\sg_i)$, we will be interested in  the points that 
lie in the four quadrants I, II, III, and IV of that 
coordinate system as pictured 
in Figure \ref{fig:basic}.  For any $a,b,c,d \in  
\mathbb{N} = \{0,1,2, \ldots \}$ and any $\sg = \sg_1 \ldots \sg_n \in S_n$, the set of all permutations of length $n$, we say that $\sg_i$ matches the 
quadrant marked mesh pattern $\MMP(a,b,c,d)$ in $\sg$ if, 
in $G(\sg)$ relative 
to the coordinate system which has the point $(i,\sg_i)$ as its  
origin, there are at least $a$ points in quadrant I, 
at least $b$ points in quadrant II, at least $c$ points in quadrant 
III, and at least $d$ points in quadrant IV.  
For example, 
if $\sg = 471569283$, the point $\sg_4 =5$  matches 
the marked mesh pattern $\MMP(2,1,2,1)$ since in $G(\sg)$ relative 
to the coordinate system with the origin at $(4,5)$,  
there are 3 points in quadrant I, 
1 point in quadrant II, 2 points in quadrant III, and 2 points in 
quadrant IV.  Note that if a coordinate 
in $\MMP(a,b,c,d)$ is 0, then there is no condition imposed 
on the points in the corresponding quadrant. 

In addition, we shall 
consider patterns  $\MMP(a,b,c,d)$ where 
$a,b,c,d \in \mathbb{N} \cup \{\emptyset\}$. Here when 
a coordinate of $\MMP(a,b,c,d)$ is the empty set, then for $\sg_i$ to match  
$\MMP(a,b,c,d)$ in $\sg = \sg_1 \ldots \sg_n \in S_n$, 
it must be the case that there are no points in $G(\sg)$ relative 
to the coordinate system with the origin at $(i,\sg_i)$ in the corresponding 
quadrant. For example, if $\sg = 471569283$, the point 
$\sg_3 =1$ matches 
the marked mesh pattern $\MMP(4,2,\emptyset,\emptyset)$ since in 
$G(\sg)$ relative 
to the coordinate system with the origin at $(3,1)$, 
there are 6 points in $G(\sg)$ in quadrant I, 
2 points in $G(\sg)$ in quadrant II, no  points in 
both quadrants III  and IV.   We let 
$\mmp^{(a,b,c,d)}(\sg)$ denote the number of $i$ such that 
$\sg_i$ matches $\MMP(a,b,c,d)$ in~$\sg$.

\fig{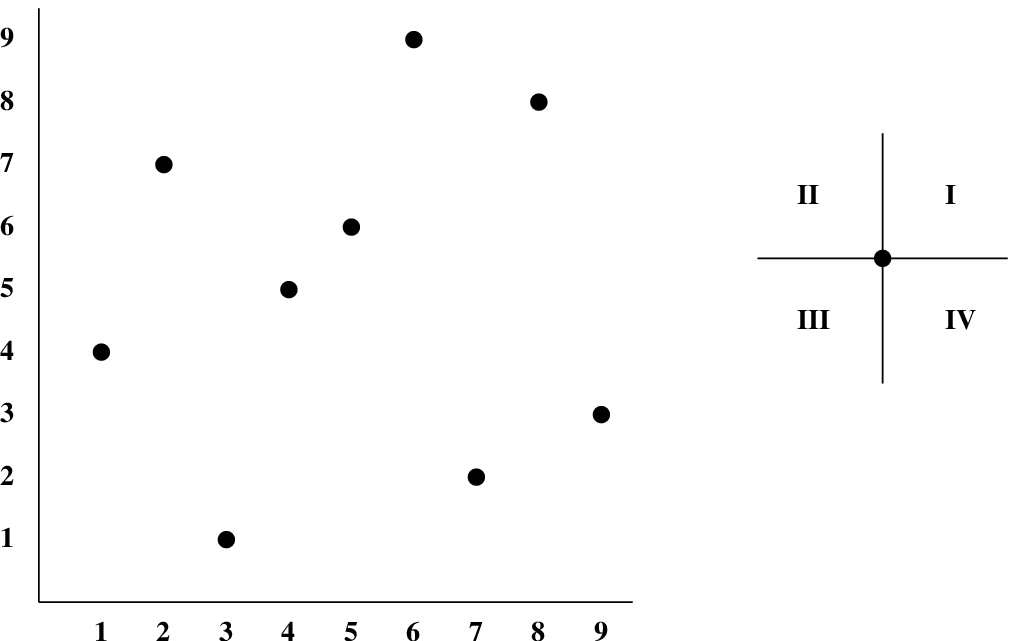}{The graph of $\sg = 471569283$.}

Note how the (two-dimensional) notation of \'Ulfarsson \cite{Ulf} for marked mesh patterns corresponds to our (one-line) notation for quadrant marked mesh patterns. For example,

\[
\MMP(0,0,k,0)=\mmpattern{scale=2.3}{1}{1/1}{}{0/0/1/1/k}\hspace{-0.25cm},\  \MMP(k,0,0,0)=\mmpattern{scale=2.3}{1}{1/1}{}{1/1/2/2/k}\hspace{-0.25cm},
\]

\[
\MMP(0,a,b,c)=\mmpattern{scale=2.3}{1}{1/1}{}{0/1/1/2/a} \hspace{-2.07cm} \mmpattern{scale=2.3}{1}{1/1}{}{0/0/1/1/b} \hspace{-2.07cm} \mmpattern{scale=2.3}{1}{1/1}{}{1/0/2/1/c} \ \mbox{ and }\ \ \ \MMP(0,0,\emptyset,k)=\mmpattern{scale=2.3}{1}{1/1}{0/0}{1/0/2/1/k}\hspace{-0.25cm}.
\]

Given a sequence $w = w_1 \ldots w_n$ of distinct integers,
let $\red[w]$ be the permutation found by replacing the
$i$-th largest integer that appears in $\sg$ by $i$.  For
example, if $\sg = 2754$, then $\red[\sg] = 1432$.  Given a
permutation $\tau=\tau_1 \ldots \tau_j$ in the symmetric group $S_j$, we say that the pattern $\tau$ {\em occurs} in $\sg = \sg_1 \ldots \sg_n \in S_n$ provided   there exists 
$1 \leq i_1 < \cdots < i_j \leq n$ such that 
$\red[\sg_{i_1} \ldots \sg_{i_j}] = \tau$.   We say 
that a permutation $\sg$ {\em avoids} the pattern $\tau$ if $\tau$ does not 
occur in $\sg$. Let $S_n(\tau)$ denote the set of permutations in $S_n$ 
which avoid $\tau$. In the theory of permutation patterns, $\tau$ is called a {\em classical pattern}. See \cite{kit} for a comprehensive introduction to 
the study of patterns in permutations. 

It has been a rather popular direction of research in the literature on permutation patterns to study permutations avoiding a 3-letter pattern subject to extra restrictions (see \cite[Subsection 6.1.5]{kit}). In \cite{kitremtie},
we started the study of the generating functions 
\begin{equation*} \label{Rabcd}
Q_{132}^{(a,b,c,d)}(t,x) = 1 + \sum_{n\geq 1} t^n  Q_{n,132}^{(a,b,c,d)}(x)
\end{equation*}
where for  any $a,b,c,d \in \{\emptyset\} \cup \mathbb{N}$, 
\begin{equation*} \label{Rabcdn}
Q_{n,132}^{(a,b,c,d)}(x) = \sum_{\sg \in S_n(132)} x^{\mmp^{(a,b,c,d)}(\sg)}.
\end{equation*}
For any $a,b,c,d$, we will write $Q_{n,132}^{(a,b,c,d)}(x)|_{x^k}$ for 
the coefficient of $x^k$ in $Q_{n,132}^{(a,b,c,d)}(x)$.

There is one obvious symmetry in this case which is induced 
by the fact that if $\sg \in S_n(132)$, then $\sg^{-1} \in S_n(132)$. 
That is, the following lemma was proved in  \cite{kitremtie}.

\begin{lemma}\label{sym} {\rm (\cite{kitremtie})}
For any $a,b,c,d \in \{\emptyset\} \cup \mathbb{N}$, 
\begin{equation*}
Q_{n,132}^{(a,b,c,d)}(x) = Q_{n,132}^{(a,d,c,b)}(x). 
\end{equation*}
\end{lemma}

In \cite{kitremtie}, we studied the generating 
functions $Q_{132}^{(k,0,0,0)}(t,x)$,  
$Q_{132}^{(0,k,0,0)}(t,x) = Q_{132}^{(0,0,0,k)}(t,x)$, and 
$Q_{132}^{(0,0,k,0)}(t,x)$  where $k$ can be either 
the empty set or a positive integer as well as the 
generating functions $Q_{132}^{(k,0,\emptyset,0)}(t,x)$ and 
$Q_{132}^{(\emptyset,0,k,0)}(t,x)$. We also showed 
that sequences of the form $(Q_{n,132}^{(a,b,c,d)}(x)|_{x^r})_{n \geq s}$ 
count a variety of combinatorial objects that appear 
in the {\em On-line Encyclopedia of Integer Sequences} (OEIS) \cite{oeis}.
Thus, our results gave new combinatorial interpretations 
of certain classical sequences such as the Fine numbers and the Fibonacci 
numbers  as well as provided certain sequences that appear in the OEIS 
with a combinatorial interpretation where none had existed before. Another particular result of our studies in \cite{kitremtie} is enumeration of permutations avoiding simultaneously the patterns 132 and 1234.

The main goal of this paper is to continue the study of 
 $Q_{132}^{(a,b,c,d)}(t,x)$ and combinatorial interpretations of 
sequences of the form 
$(Q_{n,132}^{(a,b,c,d)}(x)|_{x^r})_{n \geq s}$ in 
the case where $a,b,c,d \in \mathbb{N}$ and exactly two of 
these parameters are non-zero. The case when at least three of the parameters are non-zero will be studied in \cite{kitremtieIII}.

Next we list several 
results from \cite{kitremtie} which we need in this paper.  

\begin{theorem}\label{thm:Qk000} (\cite[Theorem 4]{kitremtie})
\begin{equation*}\label{eq:Q0000}
Q_{132}^{(0,0,0,0)}(t,x) =  C(xt) = \frac{1-\sqrt{1-4xt}}{2xt}
\end{equation*}
and, for $k \geq 1$, 
\begin{equation*}\label{Qk000}
Q_{132}^{(k,0,0,0)}(t,x) = \frac{1}{1-tQ_{132}^{(k-1,0,0,0)}(t,x)}.
\end{equation*}
Hence 
\begin{equation*}\label{eq:Q100(0)}
Q_{132}^{(1,0,0,0)}(t,0) = \frac{1}{1-t}
\end{equation*}
and, for $k \geq 2$, 
\begin{equation}\label{x=0Qk000}
Q_{132}^{(k,0,0,0)}(t,0) = \frac{1}{1-tQ_{132}^{(k-1,0,0,0)}(t,0)}.
\end{equation}
\end{theorem}

\begin{theorem}\label{thm:Q00k0} (\cite[Theorem 8]{kitremtie}) 
For $k \geq 1$, 
\begin{align}\label{gf00k0}
Q_{132}^{(0,0,k,0)}(t,x)&=\frac{1+(tx-t)(\sum_{j=0}^{k-1}C_jt^j) - 
\sqrt{(1+(tx-t)(\sum_{j=0}^{k-1}C_jt^j))^2 -4tx}}{2tx}\nonumber\\
&=\frac{2}{1+(tx-t)(\sum_{j=0}^{k-1}C_jt^j) + \sqrt{(1+(tx-t)(\sum_{j=0}^{k-1}C_jt^j))^2 -4tx}}\notag
\end{align}
and  
\begin{equation*}
Q_{132}^{(0,0,k,0)}(t,0) = \frac{1}{1-t(C_0+C_1 t+\cdots +C_{k-1}t^{k-1})}.
\end{equation*}
\end{theorem}

It follows from Lemma \ref{sym} that $Q_{132}^{(0,k,0,0)}(t,x) = 
Q_{132}^{(0,0,0,k)}(t,x)$ for all $k \geq 1$. Thus, our next 
theorem (obtained in \cite{kitremtie}) gives an expression for 
$Q_{132}^{(0,k,0,0)}(t,x) = Q_{132}^{(0,0,0,k)}(t,x)$.

\begin{theorem}\label{thm:Q0k00} (Theorem 11 of \cite{kitremtie}) 
\begin{equation*}\label{Q0100}
Q_{132}^{(0,1,0,0)}(t,x) = Q_{132}^{(0,0,0,1)}(t,x) = \frac{1}{1-tC(tx)}.
\end{equation*}
For $k > 1$, 
\begin{equation*}\label{Q0100-}
Q_{132}^{(0,k,0,0)}(t,x) = Q_{132}^{(0,0,0,k)}(t,x) = \frac{1+t\sum_{j=0}^{k-2} C_j t^j
(Q_{132}^{(0,k-1-j,0,0)}(t,x) -C(tx))}{1-tC(tx)}
\end{equation*}
and 
\begin{equation*}\label{x=0Q0100-}
Q_{132}^{(0,k,0,0)}(t,0) =Q_{132}^{(0,0,0,k)}(t,0) = \frac{1+t\sum_{j=0}^{k-2} C_j t^j
(Q_{132}^{(0,k-1-j,0,0)}(t,0) -1)}{1-t}.
\end{equation*}
\end{theorem}

As it was pointed out in \cite{kitremtie}, {\em avoidance} of a marked mesh pattern without quadrants containing the empty set can always be expressed in terms of multi-avoidance of (possibly many) classical patterns. Thus, among our results we will re-derive several known facts in permutation patterns theory. However, our main goals are more ambitious aimed at finding  distributions in question.

\section{$Q_{n,132}^{(k,0,\ell,0)}(x)$ where $k,\ell \geq 1$}

Throughout this paper, we shall classify the $132$-avoiding permutations 
$\sg = \sg_1 \ldots \sg_n$ by the position of $n$ 
in $\sg$. That is, let 
$S^{(i)}_n(132)$ denote the set of $\sg \in S_n(132)$ such 
that $\sg_i =n$. 

Clearly each $\sg \in  S_n^{(i)}(132)$ has the structure 
pictured in Figure \ref{fig:basic2}. That is, in the graph of 
$\sg$, the elements to the left of $n$, $A_i(\sg)$, have 
the structure of a $132$-avoiding permutation, the elements 
to the right of $n$, $B_i(\sg)$, have the structure of a 
$132$-avoiding permutation, and all the elements in 
$A_i(\sg)$ lie above all the elements in 
$B_i(\sg)$.  It is well-known that the number of $132$-avoiding 
permutations in $S_n$ is the {\em Catalan number} 
$C_n = \frac{1}{n+1} \binom{2n}{n}$ and the generating 
function for the $C_n$'s is given by 
\begin{equation*}\label{Catalan}
C(t) = \sum_{n \geq 0} C_n t^n = \frac{1-\sqrt{1-4t}}{2t}=
\frac{2}{1+\sqrt{1-4t}}.
\end{equation*}

\fig{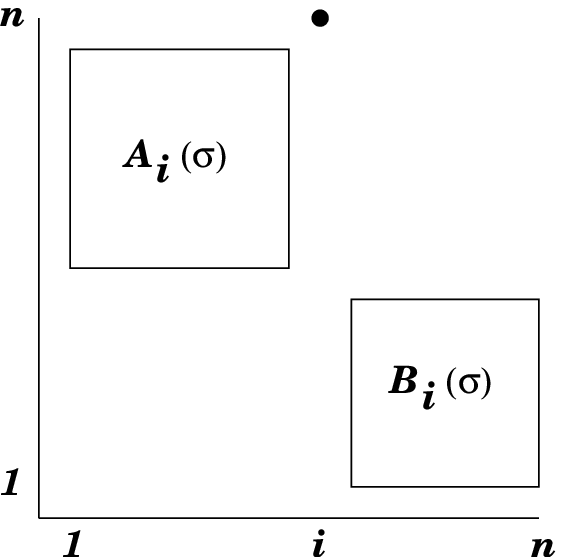}{The structure of $132$-avoiding permutations.}

If $k \geq 1$, it is easy to 
compute a recursion for $Q_{n,132}^{(k,0,\ell,0)}(x)$ for any 
fixed $\ell \geq 1$. It is clear that $n$ can never match 
the pattern $\MMP(k,0,\ell,0)$ for $k \geq 1$ in any 
$\sg \in S_n(132)$.    
For $i \geq 1$,  it is easy to see that as we sum 
over all the permutations $\sg$ in $S_n^{(i)}(132)$, our choices 
for the structure for $A_i(\sg)$ will contribute a factor 
of $Q_{i-1,132}^{(k-1,0,\ell,0)}(x)$ to $Q_{n,132}^{(k,0,\ell,0)}(x)$ since 
none of the elements to the right of $n$ have 
any effect on whether an element in  $A_i(\sg)$ matches 
the  pattern $\MMP(k,0,\ell,0)$ and the presence of $n$ ensures 
that an element in $A_i(\sg)$ matches $\MMP(k,0,\ell,0)$ in $\sg$ if 
and only if it matches $\MMP(k-1,0,\ell,0)$ in $A_i(\sg)$. 
Similarly, our choices 
for the structure for $B_i(\sg)$ will contribute a factor 
of $Q_{n-i,132}^{(k,0,\ell,0)}(x)$ to $Q_{n,132}^{(k,0,\ell,0)}(x)$ since 
neither $n$ nor any of the elements to the left of $n$ have 
any effect on whether an element in  $B_i(\sg)$ matches 
the  pattern $\MMP(k,0,\ell,0)$. 
Thus, 
\begin{equation}\label{k0l0rec}
 Q_{n,132}^{(k,0,\ell,0)}(x) =  
\sum_{i=1}^n Q_{i-1,132}^{(k-1,0,\ell,0)}(x)\ 
Q_{n-i,132}^{(k,0,\ell,0)}(x).
\end{equation}
Multiplying both sides of (\ref{k0l0rec}) by $t^n$ and summing 
over all $n \geq 1$, we obtain that 
\begin{equation*}\label{k0l0rec2}
-1+Q_{132}^{(k,0,\ell,0)}(t,x) = 
t Q_{132}^{(k-1,0,\ell,0)}(t,x)\ Q_{132}^{(k,0,\ell,0)}(t,x)
\end{equation*}
so that we have the following theorem. 

\begin{theorem}\label{thm:Qk0l0} For all $k, \ell \geq 1$, 
\begin{equation}\label{k0l0gf}
Q_{132}^{(k,0,\ell,0)}(t,x) = 
\frac{1}{1-t Q_{132}^{(k-1,0,\ell,0)}(t,x)}.
\end{equation}
\end{theorem}

Note that by Theorem \ref{thm:Q00k0}, we have an explicit formula 
for $Q_{132}^{(0,0,\ell,0)}(t,x)$ for all $\ell \geq 1$ so 
that we can then use the recursion  (\ref{k0l0gf}) to 
compute $Q_{132}^{(k,0,\ell,0)}(t,x)$ for all $k \geq 1$. 

\subsection{Explicit formulas for  $Q^{(k,0,\ell,0)}_{n,132}(x)|_{x^r}$}

Note that 
\begin{equation}\label{x=0k0l0gf}
Q_{132}^{(k,0,\ell,0)}(t,0) = 
\frac{1}{1-t Q_{132}^{(k-1,0,\ell,0)}(t,0)}.
\end{equation}
Since $Q_{132}^{(1,0,0,0)}(t,0) = Q_{132}^{(0,0,1,0)}(t,0) = \frac{1}{1-t}$,
it follows from the recursions 
(\ref{x=0Qk000}) and (\ref{x=0k0l0gf}) that 
for all $k \geq 2$,  
$Q_{132}^{(k,0,0,0)}(t,0) = Q_{132}^{(k-1,0,1,0)}(t,0)$. 
This is easy to see directly. That is, 
it is clear that if in $\sg \in S_n(132)$, $\sg_j$ matches 
$\MMP(k-1,0,1,0)$, then there is an 
$i < j$ such that $\sg_i < \sg_j$ so that 
$\sg_i$ matches $\MMP(k,0,0,0)$.  Vice versa, 
suppose that in $\sg \in S_n(132)$, $\sg_j$ matches 
$\MMP(k,0,0,0)$ where $k \geq 2$. Because 
$\sg$ is $132$-avoiding this means the elements in 
the first quadrant relative to the coordinate system with 
$(j,\sg_j)$ as the origin must be increasing. Thus, 
there exist $j < j_1 < \cdots < j_k \leq n$ such that 
$\sg_j < \sg_{j_1} < \cdots < \sg_{j_k}$ and, hence, 
$\sg_{j_1}$ matches $\MMP(k-1,0,1,0)$. 
Thus, the number of $\sg \in S_n(132)$ where 
$\mmp^{(k,0,0,0)}(\sg)=0$ is equal to the number of $\sg \in S_n(132)$ where 
$\mmp^{(k-1,0,1,0)}(\sg)=0$ for $k \geq 2$. 

In  \cite{kitremtie}, we computed the generating function 
$Q_{132}^{(k,0,0,0)}(t,0)$ for small $k$. Thus, we have that 
\begin{eqnarray*}
Q_{132}^{(2,0,0,0)}(t,0)= Q_{132}^{(1,0,1,0)}(t,0) &=&\frac{1-t}{1-2t};\\
Q_{132}^{(3,0,0,0)}(t,0)= Q_{132}^{(2,0,1,0)}(t,0)&=&\frac{1-2t}{1-3t+t^2};\\
Q_{132}^{(4,0,0,0)}(t,0)= Q_{132}^{(3,0,1,0)}(t,0)&=&\frac{1-3t+t^2}{1-4t+3t^2};\\
Q_{132}^{(5,0,0,0)}(t,0)= Q_{132}^{(4,0,1,0)}(t,0)&=&\frac{1-4t+3t^2}{1-5t+6t^2-t^3};\\
Q_{132}^{(6,0,0,0)}(t,0)= Q_{132}^{(5,0,1,0)}(t,0)&=&\frac{1-5t+6t^2-t^3}{1-6t+10t^2-4t^3}, \mbox{and}\\
Q_{132}^{(7,0,0,0)}(t,0)= Q_{132}^{(6,0,1,0)}(t,0)&=&\frac{1-6t+10t^3-4t^3}{1-7t+15t^2-10t^3+t^4}.
\end{eqnarray*}

Note that 
$Q_{132}^{(0,0,2,0)}(t,0) = \frac{1}{1-t-t^2}$ by 
Theorem \ref{thm:Q00k0}. Thus, by (\ref{x=0k0l0gf}), 
we can compute that 
\begin{eqnarray*}
Q_{132}^{(1,0,2,0)}(t,0) &=& \frac{1-t-t^2}{1-2t-t^2};\\
Q_{132}^{(2,0,2,0)}(t,0) &=& \frac{1-2t-t^2}{1-3t+t^3};\\
Q_{132}^{(3,0,2,0)}(t,0) &=& \frac{1-3t+t^3}{1-4t+2t^2+2t^3}, \ \mbox{and}\\
Q_{132}^{(4,0,2,0)}(t,0) &=& \frac{1-4t+2t^2+2t^3}{1-5t+5t^2+2t^3-t^4}.
\end{eqnarray*}
We note that $\{Q^{(1,0,2,0)}_{n,132}(0)\}_{n \geq 1}$ is the sequence 
of the Pell numbers which is A000129 in the OEIS. This result should be compared with a known fact \cite[page 250]{kit} that the avoidance of $123$, $2143$ and $3214$ simultaneously gives the Pell numbers (the avoidance of $\MMP(1,0,2,0)$ is equivalent to avoiding simultaneously $2134$ and $1234$).

\begin{problem} Find a combinatorial explanation of the fact that in $S_n$, the number of (132,2134,1234)-avoiding permutations is the same as the number of (123,2143,3214)-avoiding permutations. Can any of the known bijections between $132$- and $123$-avoiding permutations (see \cite[Chapter 4]{kit}) be of help here?\end{problem} 

The sequence $\{Q^{(2,0,2,0)}_{n,132}(0)\}_{n \geq 1}$ is sequence 
A052963 in the OEIS which has the generating function $\frac{1-t-t^2}{1-3t+t^3}$.
That is, $\frac{1-2t-t^2}{1-3t+t^3} -1 = t\frac{1-t-t^2}{1-3t+t^3}$. 
This sequence had no listed combinatorial interpretation so that we have now given a combinatorial interpretation to this sequence. 

Similarly, $Q_{132}^{(0,0,3,0)}(t,0) = \frac{1}{1-t-t^2-2t^3}$. Thus, by (\ref{x=0k0l0gf}), we can compute that 
\begin{eqnarray*}
Q_{132}^{(1,0,3,0)}(t,0) &=& \frac{1-t-t^2-2t^3}{1-2t-t^2-2t^3};\\
Q_{132}^{(2,0,3,0)}(t,0) &=& \frac{1-2t-t^2-2t^3}{1-3t-t^3+2t^4};\\
Q_{132}^{(3,0,3,0)}(t,0) &=& \frac{1-3t-t^3+2t^4}{1-4t+2t^2+4t^4}, \ \mbox{and}\\
Q_{132}^{(4,0,3,0)}(t,0) &=& \frac{1-4t+2t^2+4t^4}{1-5t+5t^2+5t^4-2t^5}.
\end{eqnarray*}

In this case, the sequence $(Q_{n,132}^{(1,0,3,0)}(0))_{n \geq 1}$ is  
sequence A077938 in the OEIS which has the generating 
function $\frac{1}{1-2t-t^2-2t^3}$. That is, 
$\frac{1-t-t^2-2t^3}{1-2t-t^2-2t^3} -1 = t\frac{1}{1-2t-t^2-2t^3}$. 
This sequence had no listed combinatorial interpretation so that we have now given a combinatorial interpretation to this sequence. 

We can also find the coefficient of the highest power of 
$x$ that occurs in $Q_{n,132}^{(k,0,\ell,0)}(x)$ for any $k,\ell \geq 1$. 
That is, it is easy to 
see that the maximum possible number of matches of $\MMP(k,0,\ell,0)$  
for a $\sg= \sg_1 \ldots \sg_n \in S_n(132)$ occurs when 
 $\sg_1 \ldots \sg_{\ell}$ is a  $132$-avoiding permutation in $S_{\ell}$ 
and $\sg_{\ell+1}\ldots \sg_{n}$ is an increasing sequence. 
Thus, we have the following theorem. 

\begin{theorem}\label{maxcoeff0k0l} 
For any $k,\ell \geq 1$ and $n \geq k +\ell +1$, 
the highest power of $x$ that occurs in $Q_{n,132}^{(k,0,\ell,0)}(x)$ 
is $x^{n-k-\ell}$ which appears with a coefficient of $C_\ell$.
\end{theorem}

Given that we have computed the generating functions 
$ Q_{132}^{(0,0,\ell,0)}(t,x)$, we can then use 
(\ref{k0l0gf}) to compute the following. 

\begin{align*}
& Q_{132}^{(1,0,1,0)}(t,x)  =1+t+2 t^2+ (4+x)t^3+\left(8+5 x+x^2\right)t^4+
\left(16+17 x+8 x^2+x^3\right)t^5+\\
&\left(32+49 x+38 x^2+12 x^3+x^4\right)t^6+ 
\left(64+129 x+141 x^2+77 x^3+17 x^4+x^5\right)t^7+\\
&\left(128+321 x+453 x^2+361 x^3+143 x^4+23 x^5+x^6\right)t^8+\\
&\left(256+769 x+1326 x^2+1399 x^3+834 x^4+247 x^5+30 x^6+x^7\right)t^9+
\cdots. 
\end{align*}

\vspace{-0.5cm}

\begin{align*}
&Q_{132}^{(2,0,1,0)}(t,x)  =1+t+2 t^2+5 t^3+(13+x)t^4 +
\left(34+7 x+x^2\right)t^5+\ \ \ \ \ \ \ \ \ \ \ \ \ \ \ \ \ \ \ \ \ \ \ \ \\
&\left(89+32 x+10 x^2+x^3\right)t^6+ 
\left(233+122 x+59 x^2+14 x^3+x^4\right)t^7+ \\
& \left(610+422 x+272 x^2+106 x^3+19 x^4+x^5\right)t^8+\\
& \left(1597+1376 x+1090 x^2+591 x^3+182 x^4+25 x^5+x^6\right)t^9 + 
\cdots. 
\end{align*}

\vspace{-0.5cm}

\begin{align*}
&Q_{132}^{(3,0,1,0)}(t,x)   =1+t+2 t^2+5 t^3+14 t^4+ (41+x)t^5+ \left(122+9 x+x^2\right)t^6+\ \ \ \ \ \ \ \ \ \ \ \\
& \left(365+51 x+12 x^2+x^3\right)t^7+\left(1094+235 x+84 x^2+16 x^3+x^4\right)t^8+ \\
&\left(3281+966 x+454 x^2+139 x^3+21 x^4+x^5\right)t^9 + \cdots.
\end{align*}

We can explain several of the coefficients that appear 
in the polynomials $Q_{n,132}^{(k,0,1,0)}(x)$ for various $k$.

\begin{theorem}
$Q_{n,132}^{(1,0,1,0)}(0) = 2^{n-1}$ 
for $n \geq 1$.  
\end{theorem}
\begin{proof}
This follows immediately from the fact that 
$Q_{132}^{(1,0,1,0)}(t,0) = \frac{1-t}{1-2t}$. We can 
also give a simple inductive proof of this fact. 
  
Clearly $Q_{1,132}^{(1,0,1,0)}(0)=1$.  Assume 
that  $Q_{k,132}^{(1,0,1,0)}(0)=2^{k-1}$ for $k < n$.  Then 
suppose that $\mmp^{(1,0,1,0)}(\sg) = 0$ and $\sg_i =n$.  Then 
it must be the case that the elements to the left of $\sg_i$ are 
decreasing so that $\sg_1 \ldots \sg_{i-1} = (n-1)(n-2) \ldots (n-(i-1))$. 
But then the elements to the right of $\sg_i$ must form a 132-avoiding 
permutation of $S_{n-1}$ which has no occurrence of the pattern $\MMP(1,0,1,0)$.
Thus, if $i =n$, we only have one such $\sg$ and if 
$i <n$, we have $2^{n-i-1}$ choices for $\sg_{i+1} \ldots \sg_n$ by induction.
It follows that 
$$Q_{n,132}^{(1,0,1,0)}(0) = 1+ \sum_{i=1}^{n-1} 2^{i-1} = 2^{n-1}.$$
\end{proof}


The sequence $(Q_{n,132}^{(1,0,1,0)}(x)|_{x})_{n \geq 3}$ is 
the sequence A000337 in the OEIS which has the formula 
$a(n) = (n-1)2^{n-1} +1$, and the following theorem confirms this fact. 

\begin{theorem} 
For $n \geq 3$,
\begin{equation}\label{Q1010x}
Q_{n,132}^{(1,0,1,0)}(x)|_x = (n-3)2^{n-2} +1. 
\end{equation}
\end{theorem}
\begin{proof}
To prove (\ref{Q1010x}), we classify 
the $\sg = \sg_1 \ldots \sg_n \in S_n(132)$ such 
that $\mmp^{(1,0,1,0)}(\sg) =1$ according to whether 
the $\sg_i$ which matches $\MMP(1,0,1,0)$ occurs to the 
left or right of position of $n$ in $\sg$. 

First, 
suppose that $\sg_i=n$ and the $\sg_s$ matching $\MMP(1,0,1,0)$ in $\sg$ 
is such that  $s < i$. It follows that $\red[\sg_1 \ldots \sg_{i-1}]$ is 
an element of $S_{i-1}(132)$ such that $\mmp^{(0,0,1,0)} =1$.  
We proved in \cite{kitremtie} that 
$Q^{(0,0,1,0)}_{n,132}(x)|_x =\binom{n}{2}$ so that we 
have $\binom{i-1}{2}$ choices for $\sg_1 \ldots \sg_{i-1}$.  
It must be the case that 
$\mmp^{(1,0,1,0)}(\sg_{i+1} \ldots \sg_n) =0$ so that 
we have $2^{n-i-1}$ choices for $\sg_{i+1} \ldots \sg_n$. It follows 
that there are 
$\binom{n-1}{2} + \sum_{i=3}^{n-1} \binom{i-1}{2}2^{n-i-1}$ permutations 
$\sg \in S_n(132)$ where the unique element which matches 
$\MMP(1,0,1,0)$ occurs to the left of the position of $n$ in $\sg$. 

Next suppose that  $\sg = \sg_1 \ldots \sg_n \in S_n(132)$, 
 $\mmp^{(1,0,1,0)}(\sg) =1$, $\sg_i=n$ and the $\sg_s$ matching $\MMP(1,0,1,0)$ 
is such that $s > i$. Then the elements to the left of $\sg_i$ in $\sg$ must 
be decreasing and the elements to the right of $\sg_i$ in $\sg$ 
must be such that $\mmp^{(1,0,1,0)}(\sg_{i+1} \ldots \sg_n) = 1$. 
Thus, we have $1+(n-i-3)2^{n-i-2}$ choices for $\sg_{i+1} \ldots \sg_n$ 
by induction. It follows 
that there are 
$$\sum_{i=1}^{n-3} (1+(n-i-3)2^{n-i-2}) = (n-3) + \sum_{j=1}^{n-4} j2^{j+1}$$ permutations 
$\sg \in S_n(132)$ where the unique element which matches 
$\MMP(1,0,1,0)$ occurs to the right  of the position of $n$ in $\sg$. 
Thus, 
\begin{eqnarray*}
Q^{(1,0,1,0)}_{n,132}(x)|_x &=& (n-3) +  \sum_{j=1}^{n-4} j2^{j+1} + 
\binom{n-1}{2}+ \sum_{i=3}^{n-1} \binom{i-1}{2}2^{n-i-1} \\
&=& (n-3)2^{n-2} +1.
\end{eqnarray*}
Here the last equality can easily be proved by induction or be verified 
by Mathematica.
\end{proof}

We also can find explicit formulas for the second highest coefficient 
in $Q_n^{(k,0,1,0)}(x)$ for $k \geq 1$.

\begin{theorem}
\begin{equation}\label{secondk010}
Q_{n,132}^{(k,0,1,0)}(x)|_{x^{n-2-k}}= 2k+\binom{n-k}{2}
\end{equation}
for all $n \geq k+3$. 
\end{theorem}
\begin{proof}
We proceed by induction on $k$. 

First we shall prove that $Q_{n,132}^{(1,0,1,0)}(x)|_{x^{n-3}}= 2+\binom{n-1}{2}$ for 
$n \geq 4$. That is, suppose that 
$\sg = \sg_1 \ldots \sg_n \in S_n(132)$ and 
$\mmp^{(1,0,1,0)}(\sg) =n-3$.  If $\sg_1 =n$, then 
$\sg_2 \ldots \sg_n$ must be strictly increasing. Similarly, 
if $\sg_{n-1} =n$ so that $\sg_n =1$, then 
$\sg_1 \ldots \sg_{n-1}$ must be strictly increasing. 
It cannot be that $\sg_i = n$ where $1 < i < n-1$ because 
in that case the most $\MMP(1,0,1,0)$-matches that we can 
have in $\sg$ occurs when $\sg_1 \ldots \sg_i$ is 
an increasing sequence and $\sg_{i+1} \ldots \sg_n$ is 
an increasing sequence which would give us a total of $i-2 +n-i -2 = n-4$ 
matches of $\MMP(1,0,1,0)$.  Thus, the only other possibility 
is if $\sg_n =n$ in which case $\mmp^{(0,0,1,0)}(\sg_1 \ldots \sg_{n-1}) = 
n-3$.  We proved in  \cite{kitremtie} that 
$Q_{n,132}^{(0,0,1,0)}(x)|_{x^{n-2}} = \binom{n}{2}$. Thus, if 
$\sg_n =n$ we have that $\binom{n-1}{2}$ choices for 
$\sg_1 \ldots \sg_{n-1}$. It follows that 
$Q_{n,132}^{(1,0,1,0)}(x)|_{x^{n-3}} = 2 + \binom{n-1}{2}$ for $n \geq 4$.

Assume that $k \geq 2$ we have established (\ref{secondk010}) for $k-1$. 
We know that the highest power of $x$ that occurs in  $Q_{n,132}^{(k,0,1,0)}(x)$ is 
$x^{n-1-k}$ which occurs with a coefficient of 1 for $n \geq k+2$.  
Now 
$$Q_{n,132}^{(k,0,1,0)}(x)|_{x^{n-2-k}}= \sum_{i=1}^n(Q_{i-1,132}^{(k-1,0,1,0)}(x)Q_{n-i,132}^{(k,0,1,0)}(x))|_{x^{n-2-k}}.$$
Since the highest power of $x$ that occurs in 
$Q_{i-1,132}^{(k-1,0,1,0)}(x)$ is 
$x^{i-1 - 1 -k}$ and the highest power of $x$ that occurs in $Q_{n-i,132}^{(k-1,0,1,0)}(x)$ is 
$x^{n-i - 1 -k}$, $(Q_{i-1,132}^{(k-1,0,1,0)}(x)Q_{n-i,132}^{(k,0,1,0)}(x))|_{x^{n-2-k}}=0$ 
unless $i\in \{1,n-1,n\}$. Thus, we have 3 cases. \\
\ \\
{\bf Case 1.} $i=1$. In that case,
$$(Q_{i-1,132}^{(k-1,0,1,0)}(x)Q_{n-i,132}^{(k,0,1,0)}(x))|_{x^{n-2-k}}= 
Q_{n-1,132}^{(k,0,1,0)}(x)|_{x^{n-2-k}} =1.$$
{\bf Case 2.} $i=n-1$. In this case, we are considering permutations of the form 
$\sg = \sg_1 \ldots \sg_{n-2} n 1$. Then we must have 
$\mmp^{(k-1,0,1,0)}(\red[\sg_1 \ldots \sg_{n-2}]) = n-k-2= (n-2)-1-(k-1)$ so that 
there is only one choice for $\sg_1 \ldots \sg_{n-2}$. Thus, in this case,  
$$(Q_{i-1,132}^{(k-1,0,1,0)}(x)Q_{n-i,132}^{(k,0,1,0)}(x))|_{x^{n-2-k}}= 
Q_{n-2,132}^{(k-1,0,1,0)}(x))|_{x^{n-2-k}} =1.$$
{\bf Case 3.} $i=n$. In this case, 
\begin{eqnarray*}
(Q_{i-1,132}^{(k-1,0,1,0)}(x)Q_{n-i,132}^{(k,0,1,0)}(x))|_{x^{n-2-k}} &=& 
Q_{n-1,132}^{(k-1,0,1,0)}(x))|_{x^{n-2-k}} \\
&=& 2(k-1) + \binom{n-1-(k-1)}{2}\\
&=& 
2(k-1) + \binom{n-k}{2}
\end{eqnarray*}
for $n-1 \geq k-1 +3$.\\
\ \\
Thus, it follows that 
$Q_{n,132}^{(k,0,1,0)}(x)|_{x^{n-2-k}}= 2k+\binom{n-k}{2}$ for $n \geq k+3$.
\end{proof}

Similarly,  we have computed the following. 

\begin{align*}
&Q_{132}^{(1,0,2,0)}(t,x)  =1+t+2 t^2+5 t^3+(12+2 x) t^4+ \left(29+11 x+2 x^2\right) t^5+\ \ \ \ \ \ \ \ \ \\
&\left(70+45 x+15 x^2+2 x^3\right) t^6+ \left(169+158 x+81 x^2+19 x^3+2 x^4\right) t^7+\\
&\left(408+509 x+359 x^2+129 x^3+23 x^4+2 x^5\right) t^8+\\
&\left(985+1550 x+1409 x^2+700 x^3+189 x^4+27 x^5+2 x^6\right) t^9+\cdots.
\end{align*}


\begin{eqnarray*}
&&Q_{132}^{(2,0,2,0)}(t,x)  =1+t+ 2 t^2+5 t^3+14 t^4+ 
(40+2x)t^5+ \left(115+15 x+2 x^2\right)t^6+\\
&& \left(331+77 x+19 x^2+2 x^3\right)t^7+
\left(953+331 x+121 x^2+23 x^3+2 x^4\right)t^8+\\
&& \left(2744+1288 x+624 x^2+177 x^3+27 x^4+2 x^5\right)t^9+ \cdots.
\end{eqnarray*}

In this case, the sequence $(Q_{n,132}^{(2,0,2,0)}(0))_{n \geq 1}$ 
is A052963 in the OEIS  which satisfies 
the recursion $a(n) = 3a(n-1)-a(n-3)$ with $a(0)=1$, $a(1) =2$ and $a(2) =5$, and has the generating function 
$\frac{1-t-t^2}{1-3t+t^3}$.

\begin{align*}
&Q_{132}^{(3,0,2,0)}(t,x)  =1+t+2 t^2+5 t^3+14 t^4+42 t^5+ (130+2x) t^6+
\left(408+19 x+2 x^2\right)t^7+ \\
&\left(1288+117 x+23 x^2+2 x^3\right)t^8+
\left(4076+588 x+169 x^2+27 x^3+2 x^4\right)t^9 + \cdots.
\end{align*}

We have also computed the following. 

\begin{eqnarray*}
&&Q_{132}^{(1,0,3,0)}(t,x)  =1+t+2 t^2+5 t^3+14 t^4+
(37+5 x)t^5+ \left(98+29 x+5 x^2\right)t^6+\\
&& \left(261+124 x+39 x^2+5 x^3\right)t^7+
\left(694+475 x+207 x^2+49 x^3+5 x^4\right)t^8+\\
&& \left(1845+1680 x+963 x^2+310 x^3+59 x^4+5 x^5\right)t^9 + \\
&& \left(4906+5635 x+4056 x^2+1692 x^3+433 x^4+69 x^5+5 x^6\right) t^{10} +\cdots.
\end{eqnarray*}

\begin{eqnarray*}
&&Q_{132}^{(2,0,3,0)}(t,x)= 1+t+2 t^2+5 t^3+14 t^4+42 t^5+(127+5 x)t^6 + 
\ \ \ \ \ \ \ \ \ \ \ \ \ \ \ \ \ \ \\
&&\left(385+39 x+5 x^2\right)t^7+\left(1169+207 x+49 x^2+5 x^3\right)t^8+ \\
&&\left(3550+938 x+310 x^2+59 x^3+5 x^4\right)t^9 + \\
&&  \left(10781+3866 x+1642 x^2+433 x^3+69 x^4+5 x^5\right) t^{10} + \cdots.
\end{eqnarray*}


\begin{eqnarray*}
&&Q_{132}^{(3,0,3,0)}(t,x)= 1+t+2 t^2+5 t^3+14 t^4+
42 t^5+132 t^6+ (424+5 x)t^7+\ \ \ \ \ \ \ \ \\
&& \left(1376+49 x+5 x^2\right)t^8+ \left(4488+310 x+59 x^2+5 x^3\right)t^9 
+ \\
&& \left(14672+1617
x+433 x^2+69 x^3+5 x^4\right)  t^{10} + \cdots.
\end{eqnarray*}


We can also find  a formula for the second highest coefficient 
in $Q_{n,132}^{(k,0,m,0)}(x)$ for $m \geq 2$. 

\begin{theorem}
For all $k \geq 1$,  $m \geq 2$ and $n \geq m+k+2$, 
$$Q_{n,132}^{(k,0,m,0)}(x)|_{x^{n-m-2}}=C_{m+1}+(2k+1)C_m+2C_m(n-k-m-2).$$
\end{theorem}

\begin{proof}

First we establish the base case which is when 
$k =1$ and $m \geq 2$. 
In this case, 
$$Q_{n,132}^{(1,0,m,0)}(x) = \sum_{i=1}^n 
Q_{i-1,132}^{(0,0,m,0)}(x) Q_{n-i,132}^{(1,0,m,0)}(x).$$

Since the highest power of $x$ that can appear in 
$Q_{n,132}^{(0,0,m,0)}(x)$ is $x^{n-m}$ for $n >m$ and 
the highest power of $x$ that can appear in 
$Q_{n,132}^{(1,0,m,0)}(x)$ is $x^{n-m-1}$ for $n >m+1$, it 
follows that the highest power of $x$ that appears in 
$Q_{i-1,132}^{(0,0,m,0)}(x) Q_{n-i,132}^{(1,0,m,0)}(x)$ 
will be less than $x^{n-m-2}$ for $i =2, \ldots, n-1$. Thus, 
we have three cases to consider. \\
\ \\
{\bf Case 1.} $i=1$. In this case, $Q_{i-1,132}^{(0,0,m,0)}(x) Q_{n-i,132}^{(1,0,m,0)}(x) 
=  Q_{n-1,132}^{(1,0,m,0)}(x)$ and we know that 
$$Q_{n-1,132}^{(1,0,m,0)}(x)|_{x^{n-m-2}} = C_m \ \mbox{for } n \geq m+2.$$
\ \\
{\bf Case 2.} $i=n-1$. In this case, $Q_{i-1,132}^{(0,0,m,0)}(x) Q_{n-i,132}^{(1,0,m,0)}(x) 
=  Q_{n-2,132}^{(0,0,m,0)}(x)$ and it was proved in 
\cite{kitremtie} that  
$$Q_{n-2,132}^{(0,0,m,0)}(x)|_{x^{n-m-2}} = C_m \ \mbox{for } n \geq m+2.$$
\ \\
{\bf Case 3.} $i=n$. In this case, $Q_{i-1,132}^{(0,0,m,0)}(x) Q_{n-i,132}^{(1,0,m,0)}(x) 
=  Q_{n-1,132}^{(0,0,m,0)}(x)$ and it was proved in 
\cite{kitremtie} that  
$$Q_{n-1,132}^{(0,0,m,0)}(x)|_{x^{n-m-2}} = C_{m+1}-C_m + 
2C_m(n-2-m) \ \mbox{for } n \geq m+3.$$

Thus, it follows that 
\begin{eqnarray*}
Q_{n,132}^{(1,0,m,0)}(x)|_{x^{n-m-2}} &=& C_{m+1}+C_m + 
2C_m(n-2-m) \\
&=& C_{m+1}+3C_m + 2C_m(n-3-m)\ \mbox{for } n \geq m+3.
\end{eqnarray*}

For example, for $m=2$, we get that 
$$Q_{n,132}^{(1,0,2,0)}(x)|_{x^{n-4}} = 11+4(n-5) \ \mbox{for } n \geq 5$$
and, for $m=3$, we get that 
$$Q_{n,132}^{(1,0,2,0)}(x)|_{x^{n-4}} = 29+10(n-6) \ \mbox{for } n \geq 6$$
which agrees with the series that we computed.

Now assume that $k > 1$ and we have proved the theorem 
for $k-1$ and all $m \geq 2$.  Then  
$$Q_{n,132}^{(k,0,m,0)}(x) = \sum_{i=1}^n 
Q_{i-1,132}^{(k-1,0,m,0)}(x) Q_{n-i,132}^{(k,0,m,0)}(x).$$

Since the highest power of $x$ that can appear in 
$Q_{n,132}^{(k-1,0,m,0)}(x)$ is $x^{n-m-(k-1)}$ for $n \geq m+k$ and 
the highest power of $x$ that can appear in 
$Q_{n,132}^{(k,0,m,0)}(x)$ is $x^{n-m-k}$ for $n >m+k$, it 
follows that the highest power of $x$ that appears in 
$Q_{i-1,132}^{(k-1,0,m,0)}(x) Q_{n-i,132}^{(k,0,m,0)}(x)$ 
will be less than $x^{n-m-k-1}$ for $i =2, \ldots, n-1$. Thus, 
we have three cases to consider. \\
\ \\
{\bf Case 1.} $i=1$. In this case, $Q_{i-1,132}^{(k-1,0,m,0)}(x) Q_{n-i,132}^{(k,0,m,0)}(x) 
=  Q_{n-1,132}^{(k,0,m,0)}(x)$ and we know that 
$$Q_{n-1,132}^{(k,0,m,0)}(x)|_{x^{n-m-k-1}} = C_m \ \mbox{for } n \geq m+k+2.$$
\ \\
{\bf Case 2.} $i=n-1$. In this case, $Q_{i-1,132}^{(k-1,0,m,0)}(x) Q_{n-i,132}^{(k,0,m,0)}(x) 
=  Q_{n-2,132}^{(k-1,0,m,0)}(x)$ and we know that  
$$Q_{n-2,132}^{(k-1,0,m,0)}(x)|_{x^{n-m-k-1}} = C_m \ \mbox{for } n \geq m+k+2.$$
\ \\
{\bf Case 3.} $i=n$. In this case, $Q_{i-1,132}^{(k-1,0,m,0)}(x) Q_{n-i,132}^{(k,0,m,0)}(x) 
=  Q_{n-1,132}^{(k-1,0,m,0)}(x)$ and we know by induction that  
$$Q_{n-1,132}^{(k-1,0,m,0)}(x)|_{x^{n-m-k-1}} = C_{m+1}+(2(k-1)+1)C_m + 
2C_m(n-m-(k-1)-1) \ \mbox{for } n \geq m+k+2.$$

Thus, it follows that 
$$
Q_{n,132}^{(k,0,m,0)}(x)|_{x^{n-m-k-1}} 
= C_{m+1}+(2k+1)C_m + 2C_m(n-m-k-2)\ \mbox{for } n \geq m+k+2.$$
\end{proof}

\section{$Q_{n,132}^{(k,0,0,\ell)}(x)=Q_{n,132}^{(k,\ell,0,0)}(x)$
 where $k,\ell \geq 1$}

By Lemma \ref{sym}, we know that 
$Q_{n,132}^{(k,0,0,\ell)}(x)= Q_{n,132}^{(k,\ell,0,0)}(x)$. 
Thus, we will only consider $Q_{n,132}^{(k,0,0,\ell)}(x)$ in 
this section.

Suppose that $n \geq \ell +1$.  
It is clear that $n$ can never match 
the pattern $\MMP(k,0,0,\ell)$ for $k \geq 1$ in any 
$\sg \in S_n(132)$.    
For $i \leq n-\ell$,  it is easy to see that as we sum 
over all the permutations $\sg$ in $S_n^{(i)}(132)$, our choices 
for the structure for $A_i(\sg)$ will contribute a factor 
of $Q_{i-1,132}^{(k-1,0,0,0)}(x)$ to $Q_{n,132}^{(k,0,0,\ell)}(x)$. 
That is,  since 
all the elements $A_i(\sg)$ have the elements in $B_i(\sg)$ in their 
fourth quadrant and $B_i(\sg)$ consists of at least $\ell$ elements so 
that the presence of $n$ ensures 
that an element in $A_i(\sg)$ matches $\MMP(k,0,0,\ell)$ in $\sg$ if 
and only if it matches $\MMP(k-1,0,0,0)$ in $A_i(\sg)$.  
Similarly, our choices 
for the structure for $B_i(\sg)$ will contribute a factor 
of $Q_{n-i,132}^{(k,0,0,\ell)}(x)$ to $Q_{n,132}^{(k,0,0,\ell)}(x)$ since 
neither $n$ nor any of the elements to the left of $n$ have 
any effect on whether an element in  $B_i(\sg)$ matches 
$\MMP(k,0,0,\ell)$. 

Now suppose $i > n-\ell$ and $j =n-i$. In this case, $B_i(\sg)$ consists 
of $j$ elements.  In this situation, an element 
of $A_i(\sg)$ matches $\MMP(k,0,0,\ell)$ in $\sg$ if and only if it matches $\MMP(k-1,0,0,\ell -j)$ in 
$A_i(\sg)$.  Thus, 
our choices for $A_i(\sg)$ contribute a factor of 
$Q^{(k-1,0,0,\ell-j)}_{i-1,132}(x) = Q^{(k-1,0,0,\ell -j)}_{n-j-1,132}(x)$ 
to $Q_{n,132}^{(k,0,0,\ell)}(x)$. Similarly, our choices 
for the structure for $B_i(\sg)$ will contribute a factor 
of $Q_{n-i,132}^{(k,0,0,\ell)}(x)$ to $Q_{n,132}^{(k,0,0,\ell)}(x)$ since 
neither $n$ nor any of the elements to the left of $n$ have 
any effect on whether an element in  $B_i(\sg)$ matches 
the  pattern $\MMP(k,0,0,\ell)$. Note that 
since 
$j < \ell$, we know that $Q_{n-i,132}^{(k,0,0,\ell)}(x) =C_j$.

It follows that for $n \geq \ell +1$, 
\begin{eqnarray}\label{Q-k00l}
Q_{n,132}^{(k,0,0,\ell)}(x) &=& \sum_{i=1}^{n-\ell} 
Q_{i-1,132}^{(k-1,0,0,0)}(x)Q_{n-i,132}^{(k,0,0,\ell)}(x) + \nonumber \\
&& \sum_{j=0}^{\ell -1} C_j Q_{n-j-1,132}^{(k-1,0,0,\ell-j)}(x).
\end{eqnarray}
Multiplying both sides of (\ref{Q-k00l}) by $t^n$, summing for $n \geq \ell +1$ and observing 
that $Q_{j,132}^{(k,0,0,\ell)}(x) = C_j$ for $j \leq \ell$, we 
see that for $k, \ell \geq 1$, 
\begin{eqnarray*}\label{Qk00l} 
Q_{132}^{(k,0,0,\ell)}(t,x)- \sum_{j=0}^\ell C_jt^j &=& 
t Q_{132}^{(k-1,0,0,0)}(t,x)\left(Q_{132}^{(k,0,0,\ell)}(t,x) -  \sum_{j=0}^{\ell-1} C_jt^j\right)+ 
\nonumber \\
&& t \sum_{j=0}^{\ell -1} C_j t^j \left(Q_{132}^{(k-1,0,0,0)}(t,x) - 
\sum_{s=0}^{\ell-j-1} C_st^s\right).
\end{eqnarray*}
Thus, we have the following theorem. 
\begin{theorem} \label{thm:k00l}
For all $k, \ell \geq 1$, 
\begin{multline}\label{Qk00lgf-}
Q_{132}^{(k,0,0,\ell)}(t,x) = \\
\frac{C_\ell t^\ell + \sum_{j=0}^{\ell -1} C_j t^j (1 -tQ_{132}^{(k-1,0,0,0)}(t,x)
+t(Q_{132}^{(k-1,0,0,\ell-j)}(t,x)-\sum_{s=0}^{\ell -j -1}C_s t^s))}{1-tQ_{132}^{(k-1,0,0,0)}(t,x)}.
\end{multline}
\end{theorem}

Note that we can compute generating functions 
of the form $Q_{132}^{(k,0,0,0)}(t,x)$  by Theorem~\ref{thm:Qk000} and 
generating functions of the form $Q_{132}^{(0,0,0,\ell)}(t,x)$ 
by Theorem~\ref{thm:Q0k00} so that we can use 
(\ref{Qk00lgf-})  to compute $Q_{132}^{(k,0,0,\ell)}(t,x)$ for 
any $k, \ell \geq 0$. 

\subsection{Explicit formulas for  $Q^{(k,0,0,\ell)}_{n,132}(x)|_{x^r}$}

By Theorem \ref{thm:k00l}, we have that  
\begin{eqnarray}\label{Qk00lgf--}
Q_{132}^{(k,0,0,1)}(t,x) &=&
\frac{t + (1 -tQ_{132}^{(k-1,0,0,0)}(t,x))
+t(Q_{132}^{(k-1,0,0,1)}(t,x)-1)}{1-tQ_{132}^{(k-1,0,0,0)}(t,x)} \nonumber \\
&=& \frac{1 -tQ_{132}^{(k-1,0,0,0)}(t,x)
+tQ_{132}^{(k-1,0,0,1)}(t,x)}{1-tQ_{132}^{(k-1,0,0,0)}(t,x)}.
\end{eqnarray}

We note that $Q_{132}^{(0,0,0,0)}(t,x) = C(tx)$ so that 
$Q_{132}^{(0,0,0,0)}(t,0)=1$.  As described in 
the previous section, we have computed $Q_{132}^{(k,0,0,0)}(t,0)$ 
for small values of $k$ in \cite{kitremtie}. 
Plugging those generating functions into (\ref{Qk00lgf--}), one 
can compute that 
\begin{eqnarray*}
Q_{132}^{(1,0,0,1)}(t,0) &=& \frac{1-t+t^2}{(1-t)^2},\\
Q_{132}^{(2,0,0,1)}(t,0) &=& \frac{1-2t+t^2+t^3}{1-3t+2t^2},\\
Q_{132}^{(3,0,0,1)}(t,0) &=& \frac{1-3t+2t^2+t^4}{1-4t+4t^2-t^3},\\
Q_{132}^{(4,0,0,1)}(t,0) &=& \frac{1-4t+4t^2-t^3+t^5}{1-5t+7t^2-3t^3}, \ \mbox{and}\\
Q_{132}^{(5,0,0,1)}(t,0) &=& \frac{1-5t+7t^2-3t^3+t^6}{1-6t+11t^2-7t^3+t^4}.
\end{eqnarray*}

It is easy to see that the maximum number of $\MMP{(1,0,0,1)}$-matches occurs 
when either $\sg$ ends with $1n$ or $n1$. It follows that 
for $n \geq 3$, the highest power of $x$ in $Q^{(1,0,0,1)}_{n,132}(x)$ is 
$x^{n-2}$ and its coefficient is $2C_{n-2}$. 
More generally, it is easy to see that the maximum number of 
$\MMP{(k,0,0,1)}$-matches occurs 
when $\sg \in S_n(132)$ ends with a shuffle of $1$ with 
$(n-k+1)(n-k) \ldots n$.  Thus, we have the following 
theorem.

\begin{theorem}\label{highQk001} 
For $n \geq k+1$, the highest power of $x$ in $Q^{(k,0,0,1)}_{n,132}(x)$ is 
$x^{n-k-1}$ and its coefficient is $(k+1)C_{n-k-1}$. 
\end{theorem}

We can also compute 
\begin{eqnarray*}
&&Q_{132}^{(1,0,0,1)}(t,x)= 1+t+2 t^2+(3+2 x) t^3+(4+6 x+4 x^2) t^4+\\
&&(5+12 x+15 x^2+10 x^3) t^5+(6+20 x+36 x^2+42 x^3+28 x^4) t^6+\\
&&(7+30 x+70 x^2+112 x^3+126 x^4+84 x^5) t^7+\\
&&(8+42 x+120 x^2+240 x^3+360 x^4+396 x^5+264 x^6) t^8+\\
&&(9+56 x+189 x^2+450 x^3+825 x^4+1188 x^5+1287 x^6+858 x^7) t^9+\cdots.
\end{eqnarray*}

It is easy to explain some of these coefficients.
That is, we have the following theorem. 
\begin{theorem}  \ 
\begin{itemize}
\item[(i)] $\displaystyle Q^{(1,0,0,1)}_{n,132}(0) =  n$ for all $n \geq 1$, 
\item[(ii)] $\displaystyle Q^{(1,0,0,1)}_{n,132}(x)|_x = (n-1)(n-2)$ for all $n \geq 3$, and 
\item[(iii)] $\displaystyle Q^{(1,0,0,1)}_{n,132}(x)|_{x^{n-3}} = 3C_{n-2}$ for all $n \geq 3$.
\end{itemize}
\end{theorem}
\begin{proof}
To see that $Q^{(1,0,0,1)}_{n,132}(0) =n$ for $n \geq 1$ note that  
the only permutations $\sg \in S_n(132)$ that have no 
$\MMP{(1,0,0,1)}$-matches are the identity 
$12\ldots n$ plus the permutations 
of the form $n(n-1) \ldots (n-k)12 \ldots (n-k-1)$ for $k=0,\ldots, n-1$. 

For $n \geq 3$, we claim that 
$$a(n) = Q^{(1,0,0,1)}_{n,132}(x)|_x =(n-1)(n-2).$$ 
This is easy to see by induction. That is, there are 
three ways to have a $\sg \in S_n(132)$ with $\mmp^{(1,0,0,1)}(\sg) =1$. 
That is, $\sg$ can start with $n$ in which case we have 
$a(n-1) =(n-2)(n-3)$ ways to arrange $\sg_2 \ldots \sg_n$ or 
$\sg$ can start with $(n-1)n$ in which case there can 
be no $\MMP(1,0,0,1)$ matches in $\sg_3 \ldots \sg_n$ which means 
that we have $(n-2)$ choices to arrange $\sg_3 \ldots \sg_n$ or 
$\sg$ can end with $n$ in which case $\sg_1 \ldots \sg_{n-1}$ must have exactly one $\MMP(0,0,0,1)$-match so that by our previous results 
in \cite{kitremtie}, we have 
$n-2$ ways to arrange $\sg_1 \ldots \sg_n$.  Thus, 
$a(n) = (n-2)(n-3) +2(n-2) = (n-1)(n-2)$.

For $Q^{(1,0,0,1)}_{n,132}(x)|_{x^{n-3}}$, we note that 
\begin{eqnarray*}
Q_{n,132}^{(1,0,0,1)}(x) &=& Q_{n-1,132}^{(0,0,0,1)}(x)+\sum_{i=1}^{n-1} 
Q_{i-1,132}^{(0,0,0,0)}(x)Q_{n-i,132}^{(1,0,0,1)}(x)\\
&=& Q_{n-1,132}^{(0,0,0,1)}(x)+\sum_{i=1}^{n-1} 
C_{i-1}x^{i-1}Q_{n-i,132}^{(1,0,0,1)}(x).
\end{eqnarray*}
Thus, 
$$Q_{n,132}^{(1,0,0,1)}(x)|_{x^{n-3}} = Q_{n-1,132}^{(0,0,0,1)}(x)|_{x^{n-3}}
+\sum_{i=1}^{n-2} 
C_{i-1}Q_{n-i,132}^{(0,0,0,1)}(x)|_{x^{n-i-2}}.$$

It was proved in \cite{kitremtie} that 
$Q_{n,132}^{(0,0,0,1)}(x)|_{x^{n-2}} =C_{n-1}$ for $n \geq 2$ and,
by Theorem \ref {highQk001}, \\
$Q_{n,132}^{(1,0,0,1)}(x)|_{x^{n-2}} =2C_{n-1}$ for 
$n \geq 2$.  Thus, for $n \geq 3$,
\begin{eqnarray*}
Q_{n,132}^{(1,0,0,1)}(x)|_{x^{n-3}} &=& C_{n-2} + \sum_{i=1}^{n-2} 
C_{i-1}2C_{n-i-2}\\
&=&C_{n-2}+ 2\sum_{i=1}^{n-2} 
C_{i-1}C_{n-i-2} = C_{n-2} +2C_{n-2} =3C_{n-2}.
\end{eqnarray*}
\end{proof}

One can also compute that 
\begin{align*}
&Q_{132}^{(2,0,0,1)}(t,x)= 1+t+2 t^2+5 t^3+(11+3 x) t^4+(23+13 x+6 x^2) t^5+
\ \ \ \ \ \ \ \ \ \ \ \ \ \\
&(47+40 x+30 x^2+15 x^3) t^6+(95+107 x+104 x^2+81 x^3+42 x^4) t^7+\\
&(191+266 x+308 x^2+301 x^3+238 x^4+126 x^5) t^8+\\
&(383+633 x+837 x^2+949 x^3+926 x^4+738 x^5+396 x^6) t^9+\cdots 
\end{align*}
and 
\begin{align*}
&Q_{132}^{(3,0,0,1)}(t,x) =1+t+2 t^2+5 t^3+14 t^4+
(38+4 x) t^5+(101+23 x+8 x^2) t^6+\ \ \ \ \ \\
&(266+92 x+51 x^2+20 x^3) t^7+(698+320 x+221 x^2+135 x^3+56 x^4) t^8+\\
&(1829+1038 x+821 x^2+614 x^3+392 x^4+168 x^5) t^9+\cdots.
\end{align*}

Here the sequence $(Q_{n,132}^{(2,0,0,1)}(0))_{n \geq 1}$ which 
starts out $1,2,5,11,23,47,95,191, \ldots $
is the sequence A083329 from the OEIS which counts the number 
of partitions $\pi$ of $n$, which when written 
in {\em increasing form}, is such that the permutation 
$flatten(\pi)$ avoids the permutations 213 and 312. For the  increasing form of a set partition $\pi$, one 
write the parts in increasing order separated by backslashes where 
the parts are written so that minimal elements in the parts increase.  
Then $flatten(\pi)$ is just the permutation that results 
by removing the backslashes.  For example, $\pi=13/257/468$ is written 
in increasing form and $flatten(\pi) =13257468$. 

\begin{problem} Find a bijection between the 
$\sg \in S_n(132)$ such that $\mmp^{(2,0,0,1)}(\sg) =0$ 
and the set partitions $\pi$ of $n$ such that $flatten(\pi)$ avoid 
231 and 312. 
\end{problem}

None of the sequences $(Q_{n,132}^{(k,0,0,1)}(0))_{n \geq 1}$ for 
$k=3,4,5$ appear in the OEIS.

Similarly, one can compute that 
\begin{equation*}\label{Qk002gf}
Q_{132}^{(k,0,0,2)}(t,x) =\frac{1 -(t+t^2)Q_{132}^{(k-1,0,0,0)}(t,x)+tQ_{132}^{(k-1,0,0,2)}(t,x)+
t^2Q_{132}^{(k-1,0,0,1)}(t,x)}{1-tQ_{132}^{(k-1,0,0,0)}(t,x)}.
\end{equation*}

Then one can use this formula to compute that 

\begin{align*}
&Q_{132}^{(1,0,0,2)}(t,x)= 1+t+2 t^2+5 t^3+(9+5x) t^4+(14+18 x+10 x^2) t^5+\ \ \ \ \ \\
&(20+42 x+45 x^2+25 x^3) t^6+(27+80 x+126 x^2+126 x^3+70 x^4) t^7+\\
&(35 +135x +280x^2+392x^3+378x^4+210 x^5)t^8\\
&(44+210 x+540 x^2+960 x^3+1260 x^4+1088 x^5+660 x^6) t^9+\cdots.
\end{align*}

It is easy to see that permutations $\sg \in S_n(132)$ which 
have the maximum number of $\MMP(1,0,0,2)$-matches in $\sg$ are those 
permutations that end in either $n12$, $n12$, $21n$, $2n1$ or $n21$. 
Thus, the highest power of $x$ that occurs in $Q_{n,132}^{(1,0,0,2)}(x)$ 
is $x^{n-3}$ which has a coefficient of $5C_{n-3}$.

\begin{align*}
&Q_{132}^{(2,0,0,2)}(t,x)= 1+t+2 t^2+5 t^3+14 t^4+ (33+9x) t^5 +
(72+42x+18x^2)t^6+\ \ \ \ \ \ \\
&(151+135x+98x^2+45x^3) t^7+(310+370 x+358 x^2+266 x^3+126 x^4) t^8+\\
&(629+931 x+1093 x^2+1047 x^3+784 x^4+378 x^5) t^9+\cdots.
\end{align*}

It is easy to see that permutations $\sg \in S_n(132)$ which 
have the maximum number of $\MMP(2,0,0,2)$-matches in $\sg$ are those 
permutations that end in either a shuffle of $21$ and $(n-1)n$ or  
$(n-1)n12$, $(n-1)12n$, and $12(n-1)n$.  
Thus, the highest power of $x$ that occurs in $Q_{n,132}^{(2,0,0,2)}(x)$ 
is $x^{n-4}$ which has a coefficient of $9C_{n-4}$.

\begin{align*}
&Q_{132}^{(3,0,0,2)}(t,x) = 1+t+2 t^2+5 t^3+14 t^4+42 t^5+
(118+14x) t^6+\ \ \ \ \ \ \ \ \ \ \ \ \ \ \ \ \ \ \ \ \\
&(319+82x+28x^2) t^7+(847+329x+184x^2+70x^3) t^8+\\
&(2231+1138x+807x^2+490x^3+196x^4) t^9+\cdots.
\end{align*}

It is easy to see that permutations $\sg \in S_n(132)$ which 
have the maximum number of $\MMP(2,0,0,2)$-matches in $\sg$ are those 
permutations that end in either a shuffle of $21$ and $(n-2)(n-1)n$ or  
$(n-2)(n-1)n12$, $(n-2)(n-1)12n$, $(n-2)12(n-1)n$, and $12(n-2)(n-1)n$.  
Thus, the highest power of $x$ that occurs in $Q_{n,132}^{(3,0,0,2)}(x)$ 
is $x^{n-5}$ which has a coefficient of $14C_{n-5}$.

None of the  series $(Q_{n,132}^{(k,0,0,2)}(0))_{n \geq 1}$ for 
$k=1,2,3$ appear in the OEIS.

\section{$Q_{n,132}^{(0,k,\ell,0)}(x)=Q_{n,132}^{(0,0,\ell,k)}(x)$
 where $k,\ell \geq 1$}

By Lemma \ref{sym}, we know that 
$Q_{n,132}^{(0,k,\ell,0)}(x)= Q_{n,132}^{(0,0,\ell,k)}(x)$. 
Thus, we will only consider $Q_{n,132}^{(0,k,\ell,0)}(x)$ in 
this section.

Suppose that $n \geq k$.  
It is clear that $n$ can never match 
the pattern $\MMP(0,k,\ell,0)$ for $k \geq 1$ in any 
$\sg \in S_n(132)$.    
For $i \geq k$,  it is easy to see that as we sum 
over all the permutations $\sg$ in $S_n^{(i)}(132)$, our choices 
for the structure for $A_i(\sg)$ will contribute a factor 
of $Q_{i-1,132}^{(0,k,\ell,0)}(x)$ to $Q_{n,132}^{(0,k,\ell,0)}(x)$ since 
none  of the elements to the right of $A_i(\sg)$ have any effect on whether 
an element of $A_i(\sg)$ matches $\MMP(0,k,\ell,0)$. 
The presence of $n$ and the elements of $A_i(\sg)$ ensures  
that an element in $B_i(\sg)$ matches $\MMP(0,k,\ell,0)$ in $\sg$ if 
and only if it matches $\MMP(0,0,\ell,0)$ in $B_i(\sg)$. Thus, 
our choices for $B_i(\sg)$ contribute a factor of 
$Q^{(0,0,\ell,0)}_{n-i,132}(x)$ to $Q_{n,132}^{(0,k,\ell,0)}(x)$.

Now suppose $i < k$ and $j =n-i$. In this case, $A_i(\sg)$ consists 
of $i-1$ elements.  In this situation, an element 
of $B_i(\sg)$ matches $\MMP(0,k,\ell,0)$ in $\sg$ if and only if it matches $\MMP(0,k-i,\ell,0)$ in 
$B_i(\sg)$.  Thus, 
our choices for $B_i(\sg)$ contribute a factor of 
$Q^{(0,k-i,\ell,0)}_{n-i,132}(x)$ 
to $Q_{n,132}^{(0,k,\ell,0)}(x)$. As before, our choices 
for the structure for $A_i(\sg)$ will contribute a factor 
of $Q_{i-1,132}^{(0,k,\ell,0)}(x)$ to $Q_{n,132}^{(0,k,\ell,0)}(x)$ but in 
such a situation $Q_{i-1,132}^{(0,k,\ell,0)}(x) =C_{i-1}$.

It follows that for $n \geq k$, 
\begin{eqnarray}\label{recQk00l}
Q_{n,132}^{(0,k,\ell,0)}(x) &=& \sum_{i=k}^{n} 
Q_{i-1,132}^{(0,k,\ell,0)}(x)Q_{n-i,132}^{(0,0,\ell,0)}(x) + \nonumber \\
&& \sum_{j=1}^{k-1} C_j Q_{n-j,132}^{(0,k-j,\ell,0)}(x).
\end{eqnarray}
Multiplying both sides of (\ref{recQk00l}) by $t^n$, summing for $n \geq \ell +1$ and observing 
that $Q_{j,132}^{(0,k,\ell,0)}(x) = C_j$ for $j \leq \ell$, we 
see that for $k, \ell \geq 1$, 
\begin{eqnarray*}\label{Qk00lc} 
Q_{132}^{(0,k,\ell,0)}(t,x)- \sum_{j=0}^{k-1} C_jt^j &=& 
t Q_{132}^{(0,0,\ell,0)}(t,x)\left(Q_{132}^{(0,k,\ell,0)}(t,x)-\sum_{s=0}^{k-2} C_st^s\right) + 
\nonumber \\
&&t \sum_{i=0}^{k-2} C_{i}t^{i} \left(Q_{132}^{(0,k-i-1,\ell,0)}(t,x)-\sum_{s=0}^{k-i-2} C_st^s\right).
\end{eqnarray*}
It follows that we have the following theorem. 
\begin{theorem}\label{thm:Q0kl0}
 For all $k,\ell \geq 1$, 
\begin{multline}\label{Q0kl0gf}
Q_{132}^{(0,k,\ell,0)}(t,x) = \\
\frac{C_{k-1} t^{k-1} + \sum_{j=0}^{k-2} C_j t^j \left(1 -tQ_{132}^{(0,0,\ell,0)}(t,x) 
+t(Q_{132}^{(0,k-i-1,\ell,0)}(t,x)-\sum_{s=0}^{k-i-2}C_s t^s)\right)}{1-tQ_{132}^{(0,0,\ell,0)}(t,x)}.
\end{multline}
\end{theorem}

Since we can compute $Q_{132}^{(0,0,\ell,0)}(t,x)$ by 
Theorem \ref{thm:Q00k0}, we can use (\ref{Q0kl0gf}) to 
compute \\
$Q_{132}^{(0,k,\ell,0)}(t,x)$ for all $k, \ell \geq 1$.

\subsection{Explicit formulas for  $Q^{(0,k,\ell,0)}_{n,132}(x)|_{x^r}$}

It follows from Theorem \ref{thm:Q0kl0} and  Theorem \ref{thm:Q00k0} that 
\begin{eqnarray*}
Q_{132}^{(0,1,\ell,0)}(t,0) &=& \frac{1}{1-tQ_{132}^{(0,0,\ell,0)}(t,0)}\\
&=& \frac{1}{1-t\frac{1}{1-t(C_0+C_1t+ \cdots +C_{\ell -1}t^{\ell -1})}}\\
&=& \frac{1-t(C_0+C_1t+ \cdots +C_{\ell -1}t^{\ell -1})}{1-t(1+C_0+C_1t+ \cdots +C_{\ell -1}t^{\ell -1})}.
\end{eqnarray*}
Thus, one can compute  that 
\begin{eqnarray*}
&&Q_{132}^{(0,1,1,0)}(t,0) = \frac{1-t}{1-2t};\\
&&Q_{132}^{(0,1,2,0)}(t,0) = \frac{1-t-t^2}{1-2t-t^2};\\
&&Q_{132}^{(0,1,3,0)}(t,0) = \frac{1-t-t^2-2t^3}{1-2t-t^2-2t^3}, \ \mbox{and}\\
&&Q_{132}^{(0,1,4,0)}(t,0) = \frac{1-t-t^2-2t^3-5t^4}{1-2t-t^2-2t^3-5t^4}.\\
\end{eqnarray*}

Similarly, one can compute 
\begin{equation*}
Q_{132}^{(0,2,\ell,0)}(t,x) = \frac{1-tQ_{132}^{(0,0,\ell,0)}(t,x)+tQ_{132}^{(0,1,\ell,0)}(t,x)}{1-tQ_{132}^{(0,0,\ell,0)}(t,x)} = 1+\frac{tQ_{132}^{(0,1,\ell,0)}(t,x)}{1-tQ_{132}^{(0,0,\ell,0)}(t,x)}.
\end{equation*}
Note that 
\begin{eqnarray*}
Q_{132}^{(0,2,\ell,0)}(t,0) &=& 1+ \frac{t\frac{1-t(C_0+C_1t+ \cdots +C_{\ell -1}t^{\ell -1})}{1-t(1+C_0+C_1t+ \cdots +C_{\ell -1}t^{\ell -1})}}{1-t\frac{1}{1-t(C_0+C_1t+ \cdots +C_{\ell -1}t^{\ell -1})}}\\
&=& 1+\frac{t(1-t(C_0+C_1t+ \cdots +C_{\ell -1}t^{\ell -1}))^2}{(1-t(1+C_0+C_1t+ \cdots +C_{\ell -1}t^{\ell -1}))^2}.
\end{eqnarray*}
Thus, it follows that 
\begin{eqnarray*}
&&Q_{132}^{(0,2,1,0)}(t,0) = 1+t\left(\frac{1-t}{1-2t}\right)^2;\\
&&Q_{132}^{(0,2,2,0)}(t,0) =  1+t\left(\frac{1-t-t^2}{1-2t-t^2}\right)^2;\\
&&Q_{132}^{(0,2,3,0)}(t,0) =  
1+t\left(\frac{1-t-t^2-2t^3}{1-2t-t^2-2t^3}\right)^2,\ \mbox{and}\\
&&Q_{132}^{(0,2,4,0)}(t,0) =  
1+t\left(\frac{1-t-t^2-2t^3-5t^4}{1-2t-t^2-2t^3-5t^4}\right)^2.\\
\end{eqnarray*}

One can use (\ref{Q0kl0gf}) and our previous computations for 
$Q_{132}^{(0,0,\ell,0)}(t,x)$ to compute \\
$Q_{132}^{(0,1,\ell,0)}(t,x)$.

\begin{align*}
&Q_{132}^{(0,1,1,0)}(t,x) = 1+t+2 t^2+(4+x) t^3+(8+5 x+x^2) t^4+
(16+17 x+8 x^2+x^3) t^5+\\
&(32+49 x+38 x^2+12 x^3+x^4) t^6+(64+129 x+141 x^2+77 x^3+17 x^4+x^5) t^7+\\
&(128+321 x+453 x^2+361 x^3+143 x^4+23 x^5+x^6) t^8+\\
&(256+769 x+1326 x^2+1399 x^3+834 x^4+247 x^5+30 x^6+x^7) t^9+\cdots.
\end{align*}

\begin{align*}
&Q_{132}^{(0,1,2,0)}(t,x) = 1+t+2 t^2+5 t^3+(12+2 x) t^4+
(29+11 x+2 x^2) t^5+\ \ \ \ \ \ \ \ \ \ \ \ \ \ \ \ \ \ \ \\
&(70+45 x+15 x^2+2 x^3) t^6+(169+158 x+81 x^2+19 x^3+2 x^4) t^7+\\
& (408+509 x+359 x^2+129 x^3+23 x^4+2 x^5) t^8+\\
&(985+1550 x+1409 x^2+700 x^3+189 x^4+27 x^5+2 x^6) t^9+\cdots. 
\end{align*}

\begin{align*}
&Q_{132}^{(0,1,3,0)}(t,x) = 1+t+2 t^2+5 t^3+14 t^4+ (37+5 x)t^5+ 
(98+29 x+5 x^2)t^6+\ \ \ \ \ \ \ \ \ \   \\
& (261+124 x+39 x^2+5 x^3)t^7+ (694+475 x+207 x^2+49 x^3+5 x^4)t^8+\\
&(1845+1680 x+963 x^2+310 x^3+59 x^4+5 x^5)t^9 + \cdots. 
\end{align*}

\begin{align*}
&Q_{132}^{(0,1,4,0)}(t,x) =1+t+2 t^2+5 t^3+14 t^4+42 t^5+(118+14 x)t^6+ 
\ \ \ \ \ \ \ \ \ \ \ \ \ \ \ \ \ \ \ \ \ \ \ \ \ \\
& (331+84 x+14 x^2)t^7+ (934+370 x+112 x^2+14 x^3)t^8+ \\
&(2645+1455 x+608 x^2+140 x^3+14 x^4)t^9 + \cdots.
\end{align*}

We can explain the highest and second highest coefficients 
that appear in $Q_{n,132}^{(0,1,\ell,0)}(x)$ for all 
$\ell \geq 1$. That is, we have the following theorem. 

\begin{theorem}\label{h2ndhQ01l0} \ 
\begin{itemize}
\item[(i)] For all $\ell \geq 1$ and $n \geq \ell +1$, the highest power of $x$ in 
$Q_{n,132}^{(0,1,\ell,0)}(x)$ 
is $x^{n -\ell -1}$ and its coefficient is $C_\ell$.

\item[(ii)] $Q_{n,132}^{(0,1,1,0)}(x)|_{x^{n-3}} = 2+\binom{n-1}{2}$  for 
all $n \geq 4$. 

\item[(iii)] For all $\ell \geq 2$, 
$Q_{n,132}^{(0,1,\ell,0)}(x)|_{x^{n-\ell -2}} = 
C_{\ell+1} +C_{\ell} +2C_{\ell}(n-2 -\ell)$  for 
all $n \geq 3+ \ell$. 
\end{itemize}
\end{theorem}

\begin{proof}
For (i), it is easy to see that the maximum 
number of $\MMP(0,1,\ell,0)$ matches occurs for a $\sg \in S_n(132)$ if 
$\sg$ starts with $n$ followed by any arrangement of 
$S_{\ell}(132)$ followed by $\ell +1, \ell +2, \ldots, n-1$ in 
increasing order. 
Thus, the highest power of $x$ in $Q_{132}^{(0,1,\ell,0)}(t,x)$ 
is $x^{n -\ell -1}$ and its coefficient is $C_\ell$.

For parts (ii) and (iii), we use the fact 
that 
\begin{equation*}
Q_{n,132}^{(0,1,\ell,0)}(x) = 
\sum_{i=1}^n Q_{i-1,132}^{(0,1,\ell,0)}(x) Q_{n-i,132}^{(0,0,\ell,0)}(x).
\end{equation*}
It was proved in \cite{kitremtie} that for $n > \ell$, 
the highest power of $x$ that occurs in 
$ Q_{n,132}^{(0,0,\ell,0)}(x)$ is $x^{n-\ell}$ and its coefficient 
is $C_{\ell}$.  It follows that for $3 \leq i \leq n-1$, the highest power of $x$ 
that appears in 
$Q_{i-1,132}^{(0,1,\ell,0)}(x) Q_{n-i,132}^{(0,0,\ell,0)}(x)$ is 
less than $n -\ell -2$. Thus, we have three cases to consider. \\
\ \\
{\bf Case 1.} $i=1$. In this case 
$Q_{i-1,132}^{(0,1,\ell,0)}(x) Q_{n-i,132}^{(0,0,\ell,0)}(x) = 
Q_{n-1,132}^{(0,0,\ell,0)}(x)$ so that we get a contribution of  $Q_{n-1,132}^{(0,0,\ell,0)}(x)|_{x^{n-\ell -2}}$.\\
\ \\
{\bf Case 2.} $i=2$. In this case 
$Q_{i-1,132}^{(0,1,\ell,0)}(x) Q_{n-i,132}^{(0,0,\ell,0)}(x) = 
Q_{n-2,132}^{(0,0,\ell,0)}(x)$ so that we get a contribution of 
$Q_{n-2,132}^{(0,0,\ell,0)}(x)|_{x^{n-\ell -2}}= C_{\ell}$.\\
\ \\
{\bf Case 3.} $i=n$. In this case 
$Q_{i-1,132}^{(0,1,\ell,0)}(x) Q_{n-i,132}^{(0,0,\ell,0)}(x) = 
Q_{n-1,132}^{(0,1,\ell,0)}(x)$ so that we get a contribution of 
$Q_{n-1,132}^{(0,1,\ell,0)}(x)|_{x^{n-\ell -2}}= C_{\ell}$.\\

Thus, it follows that 
\begin{equation*}
Q_{n,132}^{(0,1,\ell,0)}(x)_{x^{n-\ell -2}} = 2C_{\ell} + 
Q_{n-1,132}^{(0,0,\ell,0)}(x)|_{x^{n-\ell -2}}.
\end{equation*}

Then parts (ii) and (iii) follow from the fact that it was 
proved in \cite{kitremtie} that 
\begin{itemize}
\item[] $Q_{n,132}^{(0,0,1,0)}(x)|_{x^{n-2}} = \binom{n}{2}$ 
for $n \geq 2$ and,  for all $k \geq 2$,
\item[] 
$Q_{n,132}^{(0,0,k,0)}(x)|_{x^{n-k-1}} = C_{k+1}-C_{k}+ 
2(n-k-1)C_{k}$ for $n \geq k+1$.
\end{itemize}
\end{proof}

Similarly, one can compute the following.

\begin{align*}
&Q_{132}^{(0,2,1,0)}(t,x) = 1+t+2 t^2+5 t^3+  (12+2x)t^4+ (24+12x+2x^2)t^5 + 
\ \ \ \ \ \ \ \ \ \ \ \ \ \ \ \ \ \\
&(64+48x+18x^2+2x^3)t^6+(144+160 x+97 x^2+26 x^3+2 x^4)t^7+\\
&(320+480x+408x^2+184x^3+36x^4+2x^5)t^8+\\
& (704+1344 x+1479 x^2+958 x^3+327 x^4+48 x^5+2 x^6)t^9 + \cdots. 
\end{align*}

\begin{align*}
&Q_{132}^{(0,2,2,0)}(t,x) =  1+t+2 t^2+5 t^3+14 t^4+ (38+4 x)t^5+(102+26x+4x^2)t^6+\ \ \ \ \ \ \\
&(271+120 x+34 x^2+4 x^3) t^7 +(714+470 x+200 x^2+42 x^3+4 x^4)t^8+\\
&(1868+1672x+964x^2+304x^3+50x^4+4x^5)t^9 + \cdots. 
\end{align*}

\begin{align*}
&Q_{132}^{(0,2,3,0)}(t,x) = 1+t+2 t^2+5 t^3+14 t^4+42 t^5+(122+10 x)t^6+ 
\ \ \ \ \ \ \ \ \ \ \ \ \ \ \ \ \ \ \ \ \ \ \\
&(351+68 x+10 x^2)t^7+(1006+326x+88x^2+10x^3)t^8 + \\
&(2168+1364x+512x^2+108x^3+10x^4)t^9 + \cdots.
\end{align*}

\begin{align*}
&Q_{132}^{(0,2,4,0)}(t,x) = 1+t+2 t^2+5 t^3+14 t^4+42 t^5+132 t^6+\\
& (401+28 x)t^7 +(1206 +196x+28x^2)t^8 + (3618+964 x+252 x^2+28 x^3)t^9 + \cdots.
\end{align*}

In this case, we can explicitly calculate the highest and 
second highest coefficients that appear in 
$Q_{n,132}^{(0,2,\ell,0)}(x)$ for sufficiently large $n$. That is, 
we have the following theorem. 

\begin{theorem}\label{h2ndhQ02l0} \ 
\begin{itemize}
\item[(i)] For all $\ell \geq 1$ and $n \geq 3+\ell$, the highest power of $x$ that appears in $Q_{n,132}^{(0,2,\ell,0)}(x)$ is $x^{n-2-\ell}$ which appears 
with a coefficient of $2C_\ell$.  

\item[(ii)] For all $n \geq 5$, $\displaystyle Q_{n,132}^{(0,2,1,0)}(x)|_{x^{n-4}} = 6 + 2\binom{n-2}{2}$.

\item[(iii)] For all $\ell \geq 2$ and $n \geq 4+ \ell$, 
$\displaystyle Q_{n,132}^{(0,2,\ell,0)}(x)|_{x^{n-3-\ell}} =2C_{\ell +1} + 8C_{\ell} + 
4C_{\ell}(n-4-\ell)$. 
\end{itemize}
\end{theorem}

\begin{proof}

For (i), it is easy to see that the maximum 
number of $\MMP(0,1,\ell,0)$-matches occurs for a $\sg \in S_n(132)$ if 
$\sg$ starts with $(n-1)n$ or $n(n-1)$ followed by any arrangement of 
$S_{\ell}(132)$ followed by $\ell +1, \ell +2, \ldots, n-2$ in 
increasing order. Thus, the highest power of $x$ $Q_{132}^{(0,2,\ell,0)}(t,x)$ 
is $x^{n -\ell -2}$ and its coefficient is $2C_\ell$.

For parts (ii) and (iii), we use the fact 
that 
\begin{equation*}
Q_{n,132}^{(0,2,\ell,0)}(x) = Q_{n-1,132}^{(0,1,\ell,0)}(x)+
\sum_{i=2}^n Q_{i-1,132}^{(0,2,\ell,0)}(x) Q_{n-i,132}^{(0,0,\ell,0)}(x).
\end{equation*}
It was proved in \cite{kitremtie} that for $n > \ell$, 
the highest power of $x$ that occurs in 
$Q_{n,132}^{(0,0,\ell,0)}(x)$ is $x^{n-\ell}$ and its coefficient 
is $C_{\ell}$.  Moreover, it was proved in \cite{kitremtie} that 
$$Q_{n,132}^{(0,0,1,0)}(x)|_{x^{n-2}} = \binom{n}{2} \ \mbox{for } n \geq 2$$
and, for $\ell \geq 2$, 
$$Q_{n,132}^{(0,0,\ell,0)}(x)|_{x^{n-1-\ell}} = C_{\ell +1}-C_{\ell} 
+2C_{\ell}(n-3-\ell)  \ \mbox{for } n \geq 3+\ell.$$

It follows that for $4 \leq i \leq n-1$, the highest power of $x$ 
that appears in \\
$Q_{i-1,132}^{(0,2,\ell,0)}(x) Q_{n-i,132}^{(0,0,\ell,0)}(x)$ is 
less than $n -\ell -3$. Thus, we have four cases to consider when computing 
$Q_{n,132}^{(0,2,1,0)}(x)|_{x^{n-4}}$. 

\ \\
{\bf Case 1.} $Q_{n-1,132}^{(0,1,1,0)}(x)|_{x^{n-4}}$. In this case, 
by Theorem \ref{h2ndhQ01l0},  we have that, 
$$Q_{n-1,132}^{(0,1,1,0)}(x)|_{x^{n-4}} = 2+\binom{n-2}{2} \ \mbox{for } 
n \geq 5.$$
\ \\
{\bf Case 2.} $i=2$. In this case 
$Q_{i-1,132}^{(0,2,1,0)}(x) Q_{n-i,132}^{(0,0,1,0)}(x) = 
Q_{n-2,132}^{(0,0,1,0)}(x)$ and 
$$Q_{n-2,132}^{(0,0,1,0)}(x)|_{x^{n-2}} = \binom{n-2}{2} 
\ \mbox{for } n \geq 4.$$
{\bf Case 3.} $i=3$. In this case 
$Q_{i-1,132}^{(0,1,1,0)}(x) Q_{n-i,132}^{(0,0,1,0)}(x) = 
2Q_{n-3,132}^{(0,0,1,0)}(x)$ so that we get a contribution of 
$Q_{n-3,132}^{(0,0,1,0)}(x)|_{x^{n-4}}= 2C_{1}=2$ for $n \geq 5$.\\
\ \\
{\bf Case 4.} $i=n$. In this case 
$Q_{i-1,132}^{(0,2,1,0)}(x) Q_{n-i,132}^{(0,0,1,0)}(x) = 
Q_{n-1,132}^{(0,2,1,0)}(x)$ so that we get a contribution of 
$Q_{n-1,132}^{(0,2,1,0)}(x)|_{x^{n-4}}= 2C_1=2$ for $n \geq 5$.\\
\ \\
Thus, it follows that 
$$Q_{n,132}^{(0,2,1,0)}(x)|_{x^{n-4}}=6 +2\binom{n-2}{2} \ \mbox{for } 
n \geq 5.$$

Similarly,  we have four cases to consider when computing 
$Q_{n,132}^{(0,2,\ell,0)}(x)|_{x^{n-3-\ell}}$ for $\ell \geq 2$.\\ 
\ \\
{\bf Case 1.} $Q_{n-1,132}^{(0,1,\ell,0)}(x)|_{x^{n-3-\ell}}$. In this case,  
by Theorem \ref{h2ndhQ01l0},
we have that   
$$Q_{n-1,132}^{(0,1,\ell,0)}(x)|_{x^{n-3-\ell}} = 
C_{\ell +1} +C_{\ell} + 2C_{\ell}(n-3-\ell) 
 \ \mbox{for } n \geq 4+\ell.$$
\ \\
{\bf Case 2.} $i=2$. In this case 
$Q_{i-1,132}^{(0,2,\ell,0)}(x) Q_{n-i,132}^{(0,0,\ell,0)}(x) = 
Q_{n-2,132}^{(0,0,\ell,0)}(x)$ and 
$$Q_{n-2,132}^{(0,0,\ell,0)}(x)|_{x^{n-3-\ell}} = 
C_{\ell +1} -C_{\ell}+2C_{\ell}(n-3-\ell) 
\ \mbox{for } n \geq 3+\ell.$$
{\bf Case 3.} $i=3$. In this case 
$Q_{i-1,132}^{(0,1,\ell,0)}(x) Q_{n-i,132}^{(0,0,\ell,0)}(x) = 
2Q_{n-3,132}^{(0,0,\ell,0)}(x)$ so that we get a contribution of 
$Q_{n-3,132}^{(0,0,\ell,0)}(x)|_{x^{n-3-\ell}}= 2C_{\ell}$ for $n \geq 4+ \ell$.\\
\ \\
{\bf Case 4.} $i=n$. In this case 
$Q_{i-1,132}^{(0,2,\ell,0)}(x) Q_{n-i,132}^{(0,0,\ell,0)}(x) = 
Q_{n-1,132}^{(0,2,\ell,0)}(x)$ so that we get a contribution of 
$Q_{n-1,132}^{(0,2,1,0)}(x)|_{x^{n-3-\ell}}= 2C_\ell$ for $n \geq 4+\ell$.\\
\ \\
Thus, it follows that for $n \geq 4+\ell$, 
\begin{eqnarray*}
Q_{n,132}^{(0,2,\ell,0)}(x)|_{x^{n-3-\ell}}&=&
2C_{\ell +1} +4C_{\ell}+4C_{\ell}(n-3-\ell) \\
&=& 2C_{\ell +1} +8C_{\ell}+4C_{\ell}(n-4-\ell).
\end{eqnarray*}

For example, when $\ell =2$, we obtain that 
$$Q_{n,132}^{(0,2,\ell,0)}(x)|_{x^{n-5}} =26+ 8(n-6) \ \mbox{for } n \geq 6$$
and, when $\ell =3$, we obtain that 
$$Q_{n,132}^{(0,3,\ell,0)}(x)|_{x^{n-6}} =68+ 20(n-7) \ \mbox{for } n \geq 7$$
which agrees with the series that we computed. 
\end{proof}

\section{$Q_{n,132}^{(0,k,0,\ell)}(x)$ where $k,\ell \geq 1$}

Suppose that $n \geq k+\ell $.  
It is clear that $n$ can never match 
the pattern $\MMP(0,k,0,\ell)$ for $k \geq 1$ in any 
$\sg \in S_n(132)$.    There are three cases that we 
have to consider when dealing with the contribution of the permutations 
of $S^{(i)}_n(132)$ to $Q^{(0,k,0,\ell)}_{n,132}(x)$.\\
\ \\
{\bf Case 1.} $i \leq k-1$. 
It is easy to see that as we sum 
over all the permutations $\sg$ in $S_n^{(i)}(132)$, our choices 
for the structure for $A_i(\sg)$ will contribute a factor 
of $C_{i-1}$ to $Q_{n,132}^{(0,k,0,\ell)}(x)$ since no element in $A_i(\sg)$ can match $\MMP(0,k,0,\ell)$. The presence of $n$ plus the  
elements in $A_i(\sg)$ ensure that an element 
in $B_i(\sg)$ matches $\MMP(0,k,0,\ell)$ in $\sg$ if 
and only if it matches $\MMP(0,k-i,0,\ell)$ in $B_i(\sg)$. Hence  
our choices for $B_i(\sg)$ contribute a factor of 
$Q^{(0,k-i,0,\ell)}_{n-i,132}(x)$ to $Q_{n,132}^{(0,k,0,\ell)}(x)$.
Thus, in this case, the elements of  $S_n^{(i)}(132)$ contribute 
$C_{i-1}Q^{(0,k-i,0,\ell)}_{n-i,132}(x)$  to $Q_{n,132}^{(0,k,0,\ell)}(x)$.\\
\ \\
{\bf Case 2.} $k \leq i \leq n-\ell$. Note that in this case, there are at least $k$ elements in $A_i(\sg) \cup \{n\}$ and at least $\ell$ in $B_i(\sg)$.  
The presence of the  
elements in $B_i(\sg)$ ensure that an element 
in $A_i(\sg)$ matches $\MMP(0,k,0,\ell)$ in $\sg$ if 
and only if it matches $\MMP(0,k,0,0)$ in $A_i(\sg)$. Hence 
our choices for $A_i(\sg)$ contribute a factor of 
$Q^{(0,k,0,0)}_{n-i,132}(x)$ to $Q_{n,132}^{(0,k,0,\ell)}(x)$.

The presence of $n$ plus the  
elements in $A_i(\sg)$ ensures that an element 
in $B_i(\sg)$ matches $\MMP(0,k,0,\ell)$ in $\sg$ if 
and only if it matches $\MMP(0,0,0,\ell)$ in $B_i(\sg)$. Thus, 
our choices for $B_i(\sg)$ contribute a factor of 
$Q^{(0,0,0,\ell)}_{n-i,132}(x)$ to $Q_{n,132}^{(0,k,0,\ell)}(x)$.
Thus, in this case, the elements of  $S_n^{(i)}(132)$ contribute 
$Q^{(0,k,0,0)}_{i-1,132}(x)Q^{(0,0,0,\ell)}_{n-i,132}(x)$  to $Q_{n,132}^{(0,k,0,\ell)}(x)$.\\
\ \\
{\bf Case 3.} $i > n-\ell$. 
Let $j = n-i$ so that $j < \ell$. 
It is easy to see that as we sum 
over all the permutations $\sg$ in $S_n^{(i)}(132)$, our choices 
for the structure for $B_i(\sg)$ will contribute a factor 
of $C_{j}$ to $Q_{n,132}^{(0,k,0,\ell)}(x)$ since no element in $B_i(\sg)$ can match $\MMP(0,k,0,\ell)$. The presence of the  
elements in $B_i(\sg)$ ensures that an element 
in $A_i(\sg)$ matches $\MMP(0,k,0,\ell)$ in $\sg$ if 
and only if it matches $\MMP(0,k,0,\ell-j)$ in $A_i(\sg)$. Hence  
our choices for $A_i(\sg)$ contribute a factor of 
$Q^{(0,k,0,\ell-j)}_{n-j-1,132}(x)$ to $Q_{n,132}^{(0,k,0,\ell)}(x)$.
Thus, in this case, the elements of  $S_n^{(i)}(132)$ contribute 
$C_{j}Q^{(0,k,0,\ell-j)}_{n-j-1,132}(x)$  to $Q_{n,132}^{(0,k,0,\ell)}(x)$.\\

It follows that for $n \geq k+\ell$, 

\begin{multline} \label{Q0k0lrec}
Q^{(0,k,0,\ell)}_{n,132}(x) = \\
 \sum_{i=1}^{k-1} 
C_{i-1} Q^{(0,k-i,0,\ell)}_{n-i,132}(x) + 
\sum_{i=k}^{n-\ell} Q^{(0,k,0,0)}_{i-1,132}(x)Q^{(0,0,0,\ell)}_{n-i,132}(x) + 
\sum_{j=0}^{\ell -1} C_j Q^{(0,k,0,\ell -j)}_{n-j-1,132}(x).
\end{multline}
Multiplying both sides of (\ref{Q0k0lrec}) by $t^n$ and summing, we see 
that \\
\ \\
$\displaystyle Q_{132}^{(0,k,0,\ell)}(t,x) 
- \sum_{j=0}^{k+\ell -1}C_jt^j =  t\left(\sum_{j=0}^{k-2} C_j t^j \left(Q_{132}^{(0,k,0,\ell-j-1)}(t,x) - 
\sum_{s=0}^{k-j-2} C_s t^s\right)\right) + $ \\
\ \ \ \ $\displaystyle  t \left(Q_{132}^{(0,k,0,0)}(t,x) - 
\sum_{u=0}^{k-2} C_u t^u\right)\left(Q_{132}^{(0,0,0,\ell,)}(t,x) - 
\sum_{v=0}^{\ell-1} C_v t^v\right)+ $ \\
\ \ \ \ $\displaystyle t \left(\sum_{j=0}^{\ell-1} C_j t^j \left(Q_{132}^{(0,k,0,\ell-j)}(t,x) - 
\sum_{s=0}^{k-j-2} C_s t^s\right)\right).$\\
\ \\
Thus
\begin{multline}\label{Q0k0lgf}
Q_{132}^{(0,k,0,\ell)}(t,x) = \\
\sum_{j=0}^{k+\ell -1}C_jt^j + 
t\left(\sum_{j=0}^{k-2} C_j t^j \left(Q_{132}^{(0,k,0,\ell-j-1)}(t,x) - 
\sum_{s=0}^{k-j-2} C_s t^s\right)\right) + \\
t \left(Q_{132}^{(0,k,0,0)}(t,x) - 
\sum_{u=0}^{k-1} C_u t^u\right)\left(Q_{132}^{(0,0,0,\ell)}(t,x) - 
\sum_{v=0}^{\ell-1} C_v t^v\right)+  \\
t \left(\sum_{j=0}^{\ell-1} C_j t^j \left(Q_{132}^{(0,k,0,\ell-j)}(t,x) - 
\sum_{w=0}^{k+\ell-j-2} C_w t^w\right)\right).
\end{multline}
Note the first term of the last term on the right-hand side of 
(\ref{Q0k0lgf}) is $t(Q_{132}^{(0,k,0,\ell)}(t,x) - \sum_{w=0}^{k+\ell-2} C_w t^w)$ 
so that we can bring the term $tQ_{132}^{(0,k,0,\ell)}(t,x)$ to the other side 
and solve  
$Q_{132}^{(0,k,0,\ell)}(t,x)$ to obtain the following 
theorem. 
\begin{theorem}\label{thm:Q0k0l} For all $k,\ell \geq 1$, 
\begin{equation}\label{Q0k0lgf2}
Q_{132}^{(0,k,0,\ell)}(t,x) = \frac{\Phi_{k,\ell}(t,x)}{1-t}
\end{equation}
where \\
$\displaystyle 
\Phi_{k,\ell}(t,x) =  \sum_{j=0}^{k+\ell -1}C_jt^j -\sum_{j=0}^{k+\ell-2}C_jt^{j+1}
+t\left(\sum_{j=0}^{k-2} C_j t^j \left(Q_{132}^{(0,k,0,\ell-j-1)}(t,x) - 
\sum_{s=0}^{k-j-2} C_s t^s\right)\right) + $ \\
$\displaystyle  t \left(Q_{132}^{(0,k,0,0)}(t,x) - 
\sum_{u=0}^{k-1} C_u t^u\right)\left(Q_{132}^{(0,0,0,\ell)}(t,x) - 
\sum_{v=0}^{\ell-1} C_v t^v\right)+ $ \\
$\displaystyle t \left(\sum_{j=1}^{\ell-1} C_j t^j \left(Q_{132}^{(0,k,0,\ell-j)}(t,x) - 
\sum_{w=0}^{k+\ell-j-2} C_w t^w\right)\right)$.
\end{theorem}

Note that we can compute $Q_{132}^{(0,k,0,0)}(t,x)$ and 
$Q_{132}^{(0,0,0,\ell)}(t,x)$ by Theorem \ref{thm:Q0k00} so 
that we can use (\ref{Q0k0lgf2}) to compute $Q_{132}^{(0,k,0,\ell)}(t,x)$ 
for all $k,\ell \geq 1$.

\subsection{Explicit formulas for  $Q^{(0,k,0,\ell)}_{n,132}(x)|_{x^r}$}

It follows from Theorem \ref{thm:Q0k0l} that 
\begin{equation*}\label{Q0101}
Q_{132}^{(0,1,0,1)}(t,x) = \frac{1+tQ_{132}^{(0,1,0,0)}(t,x)(Q_{132}^{(0,0,0,1)}(t,x)-1)}{1-t},
\end{equation*}
and
\begin{multline}\label{Q0201}
Q_{132}^{(0,2,0,1)}(t,x) = \\
\frac{1+tQ_{132}^{(0,1,0,1)}(t,x)+ 
tQ_{132}^{(0,2,0,0)}(t,x)Q_{132}^{(0,0,0,1)}(t,x)-tQ_{132}^{(0,2,0,0)}(t,x)-tQ_{132}^{(0,0,0,1)}(t,x)}{1-t}.
\end{multline}

Similarly, using the fact that 
$$Q_{132}^{(0,2,0,0)}(t,x) =Q_{132}^{(0,0,0,2)}(t,x) \mbox{ and } 
Q_{132}^{(0,2,0,1)}(t,x) =Q_{132}^{(0,1,0,2)}(t,x),$$ one can show that 

\begin{multline}\label{Q0202}
Q_{132}^{(0,2,0,2)}(t,x) =\\
 \frac{1+(t+t^2)Q_{132}^{(0,2,0,1)}(t,x)+ t(Q_{132}^{(0,2,0,0)}(t,x))^2 -(2t+t^2)Q_{132}^{(0,2,0,0)}(t,x)}{1-t}.
\end{multline}

Here are the first few terms of these series. 

\begin{align*}
&Q_{132}^{(0,1,0,1)}(t,x) =1+t+2 t^2+(4+x) t^3+(7+5 x+2 x^2) t^4+(11+14 x+12 x^2+5 x^3) t^5+\\
& (16+30 x+39 x^2+33 x^3+14 x^4) t^6+
(22+55 x+95 x^2+117 x^3+98 x^4+42 x^5) t^7+\\
& (29+91 x+195 x^2+309 x^3+36x^4+306 x^5+132 x^6) t^8+\\
&(37+140 x+357 x^2+684 x^3+1028 x^4+1197 x^5+990 x^6+429 x^7) t^9+\cdots.
\end{align*}

\vspace{-0.5cm}

\begin{align*}
&Q_{132}^{(0,2,0,1)}(t,x) =1+t+2 t^2+5 t^3+(12+2 x) t^4+(25+13 x+4 x^2) t^5+
\ \ \ \ \ \ \ \ \ \ \ \ \ \ \ \ \ \ \ \ \ \ \ \\
& (46+45 x+31 x^2+10 x^3) t^6+(77+115 x+124 x^2+85 x^3+28 x^4) t^7+\\
&(120+245 x+359 x^2+370 x^3+252 x^4+84 x^5) t^8+\\
&(177+462 x+854 x^2+1159 x^3+1160 x^4+786 x^5+264 x^6) t^9+\cdots.
\end{align*}

\vspace{-0.5cm}

\begin{align*}
&Q_{132}^{(0,2,0,2)}(t,x) = 1+t+2 t^2+5 t^3+14 t^4+(38+4 x) t^5+
(91+33 x+8 x^2) t^6+
\ \ \ \ \ \ \ \ \ \ \ \ \ \ \\
&(192+139 x+78 x^2+20 x^3) t^7+(365+419 x+377 x^2+213 x^3+56 x^4) t^8+\\
&(639+1029 x+1280 x^2+1116 x^3+630 x^4+168 x^5) t^9+\cdots.
\end{align*}

It is easy to find the coefficients of the highest power of 
$x$ in $Q^{(0,k,0,\ell)}_{n,132}(x)$. That is, we have 
the following theorem. 

\begin{theorem}
For $n \geq k+\ell +1$, the highest power of $x$ that occurs in 
 $Q^{(0,k,0,\ell)}_{n,132}(x)$ is $x^{n-k -\ell}$ 
which occurs with a coefficient of $C_kC_{\ell}C_{n-k - \ell}$.
\end{theorem}
\begin{proof}
It is easy to see that the maximum number of 
$\MMP(0,k,0,\ell)$-matches occurs for a $\sg \in S_n(132)$ if 
$\sg$ starts with some $132$-avoiding 
rearrangement of $n,n-1, \ldots, n-k+1$ and 
ends with some $132$-avoiding rearrangement of $1, 2,\ldots, \ell$. 
In the middle of such a permutation, we can choose any $132$-avoiding permutation 
of $\ell +1, \ldots, n-k$. It follows that the highest power of 
$x$ which occurs in $Q^{(0,k,0,\ell)}_{n,132}(x)$ is $x^{n-k -\ell}$ 
which occurs with a coefficient of $C_kC_{\ell}C_{n-k - \ell}$.
\end{proof} 

We can also find an explicit formula for a coefficient 
of the second highest power of $x$ that occurs in $Q^{(0,1,0,1)}_{n,132}(x)$.
\begin{theorem} For $n \geq 4$, 
$$Q^{(0,1,0,1)}_{n,132}(x)|_{n-3} = 2C_{n-2} +C_{n-3}.$$
\end{theorem}
\begin{proof}
In this case, for $n \geq 3$,
\begin{equation}\label{0101rec}
Q^{(0,1,0,1)}_{n,132}(x) = Q^{(0,1,0,1)}_{n-1,132}(x) + \sum_{i=1}^{n-1} 
Q^{(0,1,0,0)}_{i-1,132}(x)Q^{(0,0,0,1)}_{n-i,132}(x).
\end{equation}

We proved in \cite{kitremtie} that for all $n \geq 0$, 
$Q^{(1,0,0,0)}_{n,132}(x) = Q^{(0,1,0,0)}_{n,132}(x) =
Q^{(0,0,0,1)}_{n,132}(x)$.  
In addition, we proved that for $n \geq 1$, the highest power of 
$x$ that occurs in $Q^{(1,0,0,0)}_{n,132}(x)$ is $x^{n-1}$ and 
$Q^{(1,0,0,0)}_{n,132}(x)|_{x^{n-1}} = C_{n-1}$ and that for 
$n \geq 2$, $Q^{(1,0,0,0)}_{n,132}(x)|_{x^{n-2}} = C_{n-1}$.
It follows that for $n \geq 4$, 
\begin{eqnarray*}
Q^{(1,0,0,0)}_{n,132}(x)|_{x^{n-3}} &=&  
 Q^{(0,1,0,1)}_{n-1,132}(x)|_{x^{n-3}} + 
Q^{(0,0,0,1)}_{n-2,132}(x)|_{x^{n-3}} + \\
&&
\sum_{i=2}^{n-1} 
Q^{(0,1,0,0)}_{i-1,132}(x)|_{x^{i-2}}Q^{(0,0,0,1)}_{n-i,132}(x)|_{x^{n-i-1}}\\
&=&C_{n-2} + C_{n-3} + \sum_{i-2}^{n-2} C_{i-2} C_{n-i-1} \\
&=&  C_{n-2} + C_{n-3} + C_{n-2} = 2C_{n-2} +C_{n-3}.
\end{eqnarray*}
\end{proof}

We can also get explicit formulas for 
$Q_{132}^{(0,1,0,1)}(t,0)$, $Q_{132}^{(0,2,0,1)}(t,0)$, and $Q_{132}^{(0,2,0,2)}(t,0)$ based 
on the fact that we know that 
\begin{eqnarray*}
Q_{132}^{(0,1,0,0)}(t,0)&=& Q_{132}^{(0,0,0,1)}(t,0)= \frac{1}{1-t}\ \mbox{and}\\
Q_{132}^{(0,2,0,0)}(t,0)&=& Q_{132}^{(0,0,0,2)}(t,0)= \frac{1-t+t^2}{(1-t)^2}.
\end{eqnarray*}
Then one can use the above formulas to compute that 
\begin{eqnarray*}
Q_{132}^{(0,1,0,1)}(t,0)&=&\frac{1-2t+2t^2}{(1-t)^3};\\
Q_{132}^{(0,2,0,1)}(t,0)&=&\frac{1-3t+4t^2-t^3+t^4}{(1-t)^4},\ \mbox{and}\\
Q_{132}^{(0,2,0,2)}(t,0)&=&\frac{1-4t+7t^2-5t^3+4t^4+2t^5}{(1-t)^5}.
\end{eqnarray*}
It is then easy to compute $Q^{(0,1,0,1)}_{n,132}(0) = 1 + \binom{n}{2}$ 
for $n \geq 2$. This is a known fact \cite[Table 6.1]{kit} since avoidance of the pattern $\MMP(0,1,0,1)$ is equivalent to avoiding the (classical) pattern 321 (thus, here we deal with avoidance of 132 and 321).

The sequence $\{Q^{(0,2,0,1)}_{n,132}(0)\}_{n \geq 1}$ is 
A116731 in the OEIS counting the number of permutations 
of length $n$ which avoid the patterns $321$, $2143$, and $3142$.

\end{document}